%% file: cpl-v23-arxib.tex
\documentclass[11pt,a4paper]{amsart}
\usepackage{amsthm, amssymb, amsmath, a4wide}
\usepackage[all]{xy} 
\usepackage{paralist} 
\usepackage{graphicx} 
\usepackage[sans]{dsfont} 
\usepackage{varioref} 
\usepackage{subfigure}
\usepackage{tikz}
\usepackage{pgf}
\usetikzlibrary{decorations.pathreplacing,calc}
\usepackage{xifthen}
\usepackage[pagebackref,colorlinks,linkcolor=red,citecolor=blue,urlcolor=blue,hypertexnames=true]{hyperref}
\usepackage{array} 
\usepackage{multirow}

\title{Mori dream surfaces associated with curves with one place at infinity}
\author{Pinaki Mondal}
\address{Weizmann Institute of Science}
\email{pinaki@math.toronto.edu}

\DeclareMathOperator{\lc}{Lc}

\input{commands-2012}

\newcommand{\lot}{\text{l.o.t.}}
\newcommand{\ldt}{\text{l.d.t.}}

\newcommand{\dpsx}[2]{{#1} \langle \langle #2 \rangle \rangle }
\newcommand{\dpsxc}{\dpsx{\cc}{x}}

\newcommand{\vecdelta}[1]{(#1)_{\vec \delta(#1)}}
\newcommand{\gabgamma}{g_{\alpha\beta\vec\gamma}}

\newcommand{\Span}[1]{\langle #1 \rangle}
\newcommand{\surface}{\hyperref[polynation]{\ensuremath{S_{pol}}}}
\newcommand{\surfaceplus}{\hyperref[polynation]{\ensuremath{S^+_{pol}}}}
\newcommand{\surfacenum}{\hyperref[curvettinition]{\ensuremath{S_{num}}}}
\newcommand{\mytype}{semi-affine}
\newcommand{\semidegree}{\hyperref[semisection]{semidegree}}
\newcommand{\ed}{\scrE_{\vec d}}

\newcommand{\pst}{P^{st}}
\newcommand{\pstx}{P^{st}(X)}
\newcommand{\Gammad}{\Gamma_{\vec d}}

\begin{document}

\begin{abstract} 
We study a class of rational surfaces (considered in \cite{campillo-piltant-lopez-cones}) associated to curves with one place at infinity and explicitly describe generators of the Cox ring and global sections of line bundles on these surfaces. In particular, we show that their Cox rings are finitely generated, i.e.\ they are Mori dream spaces. We also compute their {\em global Zariski semigroups at infinity} (consisting of line bundles which have no base points `at infinity') and {\em global Enriques semigroups} (generated by closures of curves in $\cc^2$). In particular, we show that the global Zariski semigroups at infinity and Enriques semigroups of surfaces corresponding to pencils which are {\em equisingular at infinity} are isomorphic, which answers a question of \cite{campillo-piltant-lopez-cones-surfaces}. We also give an {\em effective} algorithm to determine if a (rational) surface {\em admits systems of numerical curvettes} (these surfaces were also considered in \cite{campillo-piltant-lopez-cones}).
\end{abstract} 

\maketitle

\section{Introduction}
Pencils of curves on $\pp^2$ form a classical topic in algebraic geometry. A particular class $\scrV_1$ of pencils that appeared in numerous works (see e.g.\ \cite{moh-analytic-irreducibility}, \cite{sathaye}, \cite{bartolo-cassou-nogues}, \cite{suzuki}, \cite{bartolo-cassou-nogues2}, \cite{campillo-piltant-lopez-cones-surfaces}, \cite{galserrat-rational}, \cite{monserrat}) consists of {\em pencils at infinity} corresponding to {\em curves with one place at infinity}, i.e.\ pencils $V := \cc \Span{F,H^d}$\footnote{We write $\cc\langle F_1, \ldots, F_k \rangle$ to denote the vector space over $\cc$ spanned by $F_1, \ldots, F_k$.}, where $H$ is the equation of a straight line $L$ ({\em the line at infinity}), and $F$ is a polynomial of degree $d$ such that the curve $C := \{F = 0\}$ has {\em one place at infinity} (i.e.\ $C$ intersects $L$ only at one point $P$ and the germ of $C$ at $P$ is irreducible). \cite{campillo-piltant-lopez-cones} studied a more general class $\scrV \supseteq \scrV_1$ of pencils where $F$ is of the form $F_1^{a_1} \cdots F_k^{a_k}$, where each $F_j$, $1 \leq j \leq k$, defines a curve with one place at infinity. In this article we continue the study of \cite{campillo-piltant-lopez-cones-surfaces} and \cite{campillo-piltant-lopez-cones} on line bundles on the surfaces $X_V$ resulting from (minimal) resolutions of basepoints of pencils $V \in \scrV$ and show that they are {\em Mori dream spaces}, i.e.\ their Cox rings are finitely generated. Finite generation of Cox rings of surfaces have been extensively studied (see e.g.\ \cite{baty-pov}, \cite{hassett-tschinkel-cox}, \cite{testa-varado}, \cite{arte-sen-face}, \cite{ottem-k32}), but we would like to point out, as it was noted in \cite{campillo-piltant-lopez-cones-surfaces}, that the surfaces $X_V$ are unlike `typical examples' of Mori dream spaces in the sense that in general the anti-canonical bundles on these surfaces are far from being nef, so that the Mori theory does not apply. Cox rings of $X_V$ for a special class of $V \in \scrV$ were in fact shown to be finitely generated in \cite{galserrat-total}; however, our method is (different and) {\em explicit} (in the sense that we explicitly describe the generators of the Cox ring) and this enables us to tackle the sharper problem of characterizing some interesting classes of line bundles on these surfaces: \\

The main problem studied in \cite{campillo-piltant-lopez-cones-surfaces} was to achieve (for surfaces of the form $X_V$, $V \in \scrV_1$) global analogues of the work of Enriques and Zariski, i.e.\ to characterize line bundles on $X_V$ which are generated by global sections (Zariski), and those which have global sections supported on (the strict transform of) curves that properly intersect the line at infinity and the exceptional curves (Enriques). We completely solve this problem and give explicit description of the global Enriques and Zariski semigroups. In particular we answer the question asked in \cite{campillo-piltant-lopez-cones-surfaces} about these semigroups: 

\begin{bold-question}[{\cite{campillo-piltant-lopez-cones-surfaces}}] \label{question}
Let $V_i := \Span{F_i,H^d} \in \scrV_1$, $i = 1,2$. Identify $\pp^2 \setminus \{H=0\}$ with $\cc^2$, and assume $V_1$ and $V_2$ are {\em equisingular at infinity}, i.e.\ the curves $C_i$ defined by $F_i$ (of the same degree $d$) have equisingular germs `at infinity' i.e.\ at the the points $O_i := C_i \cap \{H=0\}$, $1 \leq i \leq 2$. Under what conditions are the global Enriques semigroups and/or global Zariski semigroups on $X_{V_i}$'s isomorphic?
\end{bold-question}
We show that the answer to Question \ref{question} is ``always'' (Corollary \ref{equisingular-corollary}). In fact our results extend to the larger class $\scrV$: if $V_1,V_2 \in \scrV$ are equisingular at infinity, then Corollary \ref{equisingular-corollary} states that
\begin{enumerate}
\item the global Enriques semigroups of $X_{V_1}$ and $X_{V_2}$ are isomorphic.
\item \label{zariski-assertion} the `Zariski semigroups at infinity' (consisting of divisors which have no base point on $X_V\setminus \cc^2$) of $X_{V_1}$ and $X_{V_2}$ are isomorphic.\\
\end{enumerate}

Pick $V = \cc\Span{F_1^{a_1} \cdots F_k^{a_k}, H^d} \in \scrV$. Then $X_V$ can be considered naturally as a compactification of $\cc^2$ (by identifying $\cc^2$ with the complement of $\{H = 0\}$ in $\pp^2$). Set $f_i := F_i|_{\cc^2}$, $1 \leq i \leq k$. As in \cite{campillo-piltant-lopez-cones} and \cite{campillo-piltant-lopez-cones-surfaces}, we show that the divisors $D_{ij}$ corresponding to {\em approximate roots} (introduced by Abhyankar and Moh \cite{abhya-moh-tschirnhausen}) of the $f_i$'s play a crucial role in the structure of line bundles on $X_V$: they essentially generate the Cox ring of $X_V$ (Theorem \ref{cox-thm} and Remark \ref{one-place-f-ij}) and {\em tropically generates} the global Enriques semigroup $\pst(X_V)$:

\begin{defn}
Given $S \subseteq \zz^k$, the {\em tropical closure} of $S$ is the smallest semigroup $\bar S \subseteq \zz^k$ containing $S$ which is also closed under taking (coordinatewise) maximum, i.e.\ if $\alpha_i := (\alpha_{i1}, \ldots, \alpha_{ik}) \in \bar S$, $1 \leq i \leq 2$, then $\max\{\alpha_1, \alpha_2\} := (\max\{\alpha_{11}, \alpha_{21}\}, \ldots, \max\{\alpha_{1k}, \alpha_{2k}\}) \in \bar S$. We say that $S$ {\em tropically generates} $T \subseteq \zz^k$ iff $T = \bar S$.
\end{defn}

\begin{thm}[follows from Corollary \ref{enricorollary} and Remark \ref{one-place-f-ij}]
Let $\Gamma_0, \ldots, \Gamma_N$ be the irreducible components of $X_V \setminus \cc^2$. Consider the corresponding identification of $\Pic X_V$ with $\zz^{N+1}$. Then $\{D_{ij}\}_{i,j}$ tropically generates $\pst(X_V) \subseteq \zz^{N+1}$.
\end{thm}

The global Zariski semigroup at infinity $\tilde P_\infty(X_V)$ can also be expressed in terms of (products of) the approximate roots of the $f_i$'s. The statements however turn out to be a bit more technical, and we refer the reader to Theorem \ref{zariski-thm} for a precise description of $\tilde P_\infty(X)$ for surfaces $X$ in a more general class $\surface$, and to Corollary \ref{one-place-cor} for a version customized for $X_V$ for $V \in \scrV_1$. \\

Our results on Cox ring and global Enriques and Zariski semigroups are valid for a class $\surface$ of surfaces which (strictly) contains all $X_V$, $V \in \scrV$. The class $\surface$ appears naturally in the study of compactifications of $\cc^2$. Indeed, if $X$ is a compactification of $\cc^2$ then to each irreducible curve $\Gamma$ `at infinity' on $X$ (i.e.\ $\Gamma$ is an irreducible component of $X_V \setminus \cc^2$) one can associate a (finite) sequence of elements in $\cc[x,x^{-1},y]$ which describe the order of vanishing of rational functions along $\Gamma$; these are called {\em key forms} associated to $C$. The key forms are global variants of {\em key polynomials} of valuations introduced by MacLane \cite{maclane-key}, and they can be used to determine {\em effectively} various properties of the associated (divisorial) valuation, see e.g.\ \cite{contractibility}, \cite{non-negative-valuation}, \cite{with-tim}. In particular, \cite[Theorem 1.8]{non-negative-valuation} uses key forms to characterize a class $\scrC$ of divisorial valuations on $\cc(x,y)$ whose {\em skewness} (introduced in \cite{favsson-tree}) can be `read from' the slope of an extremal ray in the cone of curves on an associated compactification of $\cc^2$. On the other hand, \cite{campillo-piltant-lopez-cones} introduced the class $\surfacenum$ of compactifications of $\cc^2$ which {\em admit systems of numerical curvettes}. Theorem \ref{numerical-thm} shows that for a compactification $X$ of $\cc^2$, $X \in \surfacenum$ iff the divisorial valuation associated to each curve at infinity on $X$ is in $\scrC$; in particular it gives an {\em effective} algorithm to determine if $X \in \surfacenum$. \\

The class $\surface$, which is the central topic of this article, is the subset of $\surfacenum$ for which the key forms associated to the curvess at infinity are {\em polynomials}. In particular $\surface$ contains all $X_V$, $V \in \scrV$ (Proposition \ref{pencil-surfaces}). A geometric interpretation of the class $\surface$ comes from the following observation: if $X$ is a compactification of $\cc^2$ which dominates $\pp^2$, then $X \in \surface$ iff $X$ admits a system of numerical curvettes which `essentially come from affine curves' (Theorem \ref{my-numerical-thm}). \\

The organization of this article is follows: Section \ref{backtion} and the appendix (Section \ref{appendix}) contains some background regarding key forms, {\em semidegrees} (which are negatives of divisorial valuations), and associated {\em degree-wise Puiseux series}. In Section \ref{surfacection} we introduce the classes $\surfaceplus \subseteq \surface \subseteq \surfacenum$ of surfaces, characterize them in terms of associated `systems of curvettes', and show that each $X_V$ for $V \in \scrV$ (resp.\ $V \in \scrV_1$) is a member of $\surface$ (resp.\ $\surfaceplus$). Theorem \ref{cox-thm} in Section \ref{coxection} gives an explicit description of Cox rings of surfaces in $\surface$ and in Section \ref{linection} we describe global Enriques semigroups and global Zariski semigroups at infinity of surfaces in $\surface$. Finally, in Section \ref{onection} we reformulate the results of Sections \ref{coxection} and \ref{linection} for surfaces of the form $X_V$ for $V \in \scrV_1$ in terms of the approximate roots of the associated polynomial(s).\\

This work has been done at the Weizmann Institute as an Azrieli Fellow. I am grateful to the Azrieli Foundation for the award of an Azrieli Fellowship.

\section{Background} \label{backtion}
\subsection{Semidegrees and degree-wise Puiseux series} \label{semisection}
Let $X$ be a normal complete algebraic surface containing $U := \cc^2$, and $\Gamma$ be an irreducible curve at infinity, i.e.\ an irreducible components of $X\setminus U$. Let $\delta$ be the {\em semidegree} on $\cc(X)$ determined by $\Gamma$, i.e.\ $\delta(f)$, where $f \in \cc(X)$, is the order of pole of $f$ along $\Gamma$ (in other words, $\delta$ is the negative of the {\em divisorial discrete valuation} on $\cc(X)$ associated to $\Gamma$). Choose coordinates $(x,y)$ on $U$ such that $\delta(x) > 0$ and let $\tilde p := \delta(x)$. Then there exists $\phi(x) \in \cc[x^{1/\tilde p}, x^{-1/\tilde p}]$, and a rational number $r \in \frac{1}{\tilde p}\zz$ such that $r < \ord_x(\phi)$ and
\begin{align}
\delta(f(x,y)) = \tilde p\deg_x\left(f(x,y)|_{y = \phi(x) + \xi x^{r}}\right) \label{puiseux-degree}
\end{align}
for all $f \in \cc(x,y)$, where $\xi$ is an indeterminate (\cite[Theorem 1.2]{sub2-1}). We call $\tilde \phi(x) := \phi(x) + \xi x^{r}$ the {\em generic degree-wise Puiseux series} associated to $\delta$. The field of {\em degree-wise Puiseux series} in $x$ is 
$$\dpsxc := \bigcup_{p=1}^\infty \cc((x^{-1/p})) = \left\{\sum_{j \leq k} a x^{j/p} : k,p \in \zz,\ p \geq 1 \right\},$$
where for each integer $p \geq 1$, $\cc((x^{-1/p}))$ denotes the field of Laurent series in $x^{-1/p}$. We refer to Section \ref{appendix} for some notions, e.g.\ {\em conjugacy, Puiseux pairs}, e.t.c.\ of degree-wise Puiseux series. The usual factorization of polynomials in terms of Puiseux series (see e.g.\ \cite[Section 1.5]{casas-alvero}) implies the following

\begin{thm} \label{dpuiseux-factorization}
Let $f \in \cc[x,y]$. Then there are unique (up to conjugacy) degree-wise Puiseux series $\phi_1, \ldots, \phi_k$, a unique non-negative integer $m$ and $c \in \cc^*$ such that
$$f = cx^m \prod_{i=1}^k \prod_{\parbox{1.75cm}{\scriptsize{$\psi_{ij}$ is a con\-ju\-ga\-te of $\psi_i$}}}\mkern-27mu \left(y - \psi_{ij}(x)\right)$$
\end{thm}

The following proposition, which is a straightforward corollary of \cite[Proposition 4.2]{sub2-1}, illustrates a relation between degree-wise Puiseux factorization of polynomials and the behaviour at infinity of the curves they define: 

\begin{prop} \label{puiseuxfinity}
Let $X$ be a normal algebraic surface containing $\cc^2$. Let $\Gamma_1, \ldots, \Gamma_N$ be the irreducible components of $X\setminus \cc^2$ and $\delta_j$, $1 \leq j \leq N$, be the semidegree on $\cc(X)$ induced by $\Gamma_j$. Choose coordinates $(x,y)$ on $\cc^2$ such that $\delta_j(x) > 0$ for each $j$, $1 \leq j \leq N$. For each $j$, $1 \leq j \leq N$, let $\tilde \phi_j(x,\xi):= \phi_j(x) + \xi x^{r_j}$ be the generic degree-wise Puiseux series associated to $\delta_j$. Let $f \in \cc[x,y]$ and $C$ be the closure in $X$ of the curve $\{f=0\} \subseteq \cc^2$. 
\begin{enumerate}
\item Assume $C$ is proper and $C \cap \Gamma_j \cap \Gamma_k = \emptyset$ for all $j \neq k$. Then the degree-wise Puiseux factorization of $f$ is of the form
\begin{align}
f = c\prod_{i=1}^N \prod_{j=1}^{n_i} \prod_{k=1}^{\tilde q_{ij}}\left(y - \psi_{ijk}(x)\right) \label{all-interior-factorization}
\end{align}
where $c \in \cc^*$, $n_i := |C \cap \Gamma_i|$, and $\psi_{ijk}$'s are conjugates of some $\psi_{ij}$ of the form 
$$\phi_i(x) + \xi_{ij} x^{r_i} + \ldt$$
where $\xi_{ij} \in \cc$ and $\ldt$ stands for `terms with lower degree' in $x$. 
\item Assume in addition that each $C$ intersects each $\Gamma_j$ transversally (at points where both $\Gamma_j$ and $X$ are nonsingular). Then in \eqref{all-interior-factorization}, for all $i,j$, $\tilde q_{ij} = \tilde p_i$, and $\psi_{ij} - \phi_i(x) - \xi_{ij} x^{r_i} \in \cc((x^{-1/\tilde p_i}))$. 
\end{enumerate} 
\end{prop}

\subsection{Key forms}
Let $X$, $\Gamma$, $\delta$ and $\tilde \phi(x,\xi) = \phi(x) + \xi x^{r}$ be as in Section \ref{semisection}. One can associate to $\delta$ a finite sequence of elements in $\cc[x,x^{-1},y]$ called the {\em key forms} of $\delta$ (see \cite[Definition 3.17]{contractibility}). The sequence starts with $f_0 := x, f_1 := y$, and for each $n \geq 1$, the ($n+1$)-th key form $f_{n+1}$ is an element in $\cc[f_0,f_0^{-1},f_1,f_2, \ldots, f_n]$ whose $\delta$-value is smaller than the `expected' value. An algorithm and detailed example for the computation of key forms of $\delta$ from $\phi$ and $\omega$ appears in \cite[Section 3.3]{contractibility}.

\begin{example} 
\mbox{}\\
\begin{tabular}{|l|l|}
\hline
$\phi(x) + \xi x^\omega$ & key forms\\
\hline 
$\xi x^{p/q}$ & $x,y$ \\
\hline
$cx^{p/q} + \xi x^\omega,$ $c \in \cc^*$, $q > 0$, & $x,y, y^q - c^qx^p$ \\
$p,q$ rel.\ prime integers, $\omega < \frac{p}{q}$ & \\
\hline
$x^{5/2} + x^{-3/2} + \xi x^{-5/2}$ & $x,y,y^2 - x^5, y^2 - x^5 - 2x$\\
\hline
$x^{5/2} + x^{-1} + x^{-3/2} + \xi x^{-5/2}$ & $x,y,y^2 - x^5, y^2 - x^5 - 2x^{-1}y, y^2 - x^5 - 2x^{-1}y - 2x$\\
\hline
\end{tabular}
\end{example} 
Theorem \ref{key-thm} below sums up the properties of key forms that we need. The first assertion of Theorem \ref{key-thm} follows from defining properties of key forms (\cite[Proposition 3.28]{contractibility}), the second assertion follows from \cite[Proposition 4.2]{contractibility} and the third assertion is a corollary of \cite[Theorem 1.4]{non-negative-valuation}.

\begin{thm} \label{key-thm}
Let the key forms of $\delta$ be $f_0, \ldots, f_{n}$.
\begin{enumerate}
\item Let $p$ be the polydromy order (see Definition \ref{dpuiseuxnition}) of $\phi(x)$. Then there is $\psi(x) \in \cc((x^{-1/p}))$ such that:
\begin{align} 
\psi(x)	&= \phi(x) + \text{terms with $x$-degree $\leq r$}, \notag \\
f_{n} &= \prod_{j=1}^p (y - \psi_j(x)),\quad \text{where $\psi_j$'s are the conjugates of $\psi$.}  \label{last-factorization}
\end{align}
\item \label{polynomial-key} The following are equivalent:
\begin{enumerate}
\item $f_k$ is a polynomial for all $k$, $0 \leq k \leq n$.
\item $f_{n}$ is a polynomial.
\item There exists a polynomial $f$ with degree-wise Puiseux factorization of the form:
\begin{align} 
\begin{split}
f 		&= \prod_{i=1}^k \prod_{\parbox{1.75cm}{\scriptsize{$\psi_{ij}$ is a con\-ju\-ga\-te of $\psi_i$}}}\mkern-27mu \left(y - \psi_{ij}(x)\right),\quad \text{where for each $i$,}\\
\psi_i	&= \phi(x) + \xi_{i} x^{r} + \ldt, \qquad \xi_{i} \in \cc.
\end{split} \label{delta-factorization} 
\end{align}
\end{enumerate}
\item \label{non-negative-key} The following are equivalent:
\begin{enumerate}
\item $\delta(f) \geq 0$ for all $f \in \cc[x,y]\setminus\{0\}$.
\item $\delta(f_{n}) \geq 0$.
\item $\delta(f) \geq 0$ for some $f \in \cc[x,x^{-1},y]$ which satisfies \eqref{delta-factorization}.
\end{enumerate}
\end{enumerate}
\end{thm}

\begin{prop}[{\cite[Propositions 4.2 and 4.7]{sub2-1}}] \label{2compact}
Let $\delta$ be as in Theorem \ref{key-thm}. Assume $\delta$ is {\em not} the usual degree in $(x,y)$ coordinates, Then there exists a unique compactification $X'$ of $\cc^2$ such that 
\begin{enumerate}
\item $X'$ is projective and normal.
\item $X'_\infty := X' \setminus \cc^2$ has two irreducible components $C'_1,C'_2$.
\item The semidegree on $\cc[x,y]$ corresponding to $C'_1$ and $C'_2$ are respectively $\deg$ and $\delta$. 
\end{enumerate}
All singularities of $X'$ are {\em rational} (which implies in particular that all Weil divisors are $\qq$-Cartier). Let $(C'_i,C'_j)$, $1 \leq i, j\leq 2$, denote the intersection number of $C'_i$ and $C'_j$. Then
\begin{enumerate}
\addtocounter{enumi}{3}
\item $(C'_2,C'_2) < 0$,
\item $(C'_1,C'_1) = -q\delta(f_{l+1})$, where $q$ is a positive rational number,
\item Let $D'_2 := C'_1 - \frac{(C'_1,C'_1)}{(C'_1,C'_2)}C'_2$. Then $(C'_1,D'_2) = 0$ and $(C'_2,D'_2) = 1$.
\end{enumerate}  
\end{prop}

\section{Surfaces admitting systems of numerical (\mytype) curvettes} \label{surfacection}
\begin{defn} \label{curvettinition}
\cite{campillo-piltant-lopez-cones} introduced the notion of surfaces {\em admitting systems of numerical curvettes}. Here we extend the scope of the definition: Let $X$ be a normal complete algebraic surface containing $U := \cc^2$ as a dense open subset, and let $\Gamma_1, \ldots, \Gamma_N$ be the irreducible components of $X\setminus U$. An effective Weil divisor $\Delta$ on $X$ is called a {\em numerical $\Gamma_i$-curvette}, $1 \leq i \leq N$, if $(\Delta, \Gamma'_i) >0$ and $(\Delta,\Gamma'_j)= 0$ for all $j \neq i$. We say that $X$ {\em admits a system of numerical curvettes}, or equivalently, $X \in \surfacenum$, iff there exists a numerical $\Gamma_i$-curvette for all $i$, $1 \leq i \leq N$.  
\end{defn}

\begin{thm} \label{numerical-thm}
Let $X$ be as in Definition \ref{curvettinition}. For each $j$, $1 \leq j \leq N$, let $\delta_j$ be the semidegree associated to $\Gamma_j$ and $f_j$ be the {\em last} key form of $\delta_j$. Then the following are equivalent:
\begin{enumerate}[(a)]
\item \label{numerical-assertion} $X$ admits a system of numerical curvettes.
\item \label{poly-assertion} For each $j$, $1 \leq j \leq N$, either $\delta_j(f_j) \geq 0$ or $f_j$ is a polynomial.
\end{enumerate}
\end{thm}

\begin{proof}
Choose a system of coordinates $(x,y)$ on $\cc^2 = U$ such that no $\delta_j$ is the degree in $(x,y)$-coordinates, and let $X_0 \cong \pp^2$ be the usual compactification of $U$ given by $(x,y) \into [x:y:1]$. Let $\tilde X$ be the minimal normal surface containing $U$ which dominates both $X$ and $X_0$ (via morphisms induced by identification of $U$). Then $\tilde X\setminus U$ has $N+1$ irreducible components $\tilde \Gamma_0, \ldots, \tilde \Gamma_N$, where $\tilde \Gamma_0$ (resp.\ $\tilde \Gamma_i$, $1 \leq i \leq N$,) is the strict transform of the line at infinity on $X_0$ (resp.\ $\Gamma_i$). It follows that the semidegree associated to $\tilde \Gamma_0$ (resp.\ $\tilde \Gamma_i$, $1 \leq i \leq N$,) is the degree in $(x,y)$-coordinates (resp.\ $\delta_i$). Let $\pi: \tilde X \to X$ be the morphism induced by identification of $U$. \\

\eqref{numerical-assertion} $\im$ \eqref{poly-assertion}: Assume $X$ admits a system of numerical curvettes. Fix $j$, $1 \leq j \leq N$. Let $\Delta_j := D_j + \sum_{i=1}^N a_i \Gamma_i$ be a numerical $\Gamma_j$-curvette on $X$, where $D_j$ is an effective Weil divisor such that $\supp(D_j) \not \supset \Gamma_i$ for any $i$, and $a_i \geq 0$ for all $i$, $1 \leq i \leq N$. Let $\tilde \Delta_j := \pi^*(\Delta_j)$. Then $\tilde \Delta_j$ is of the form $\tilde D_j + \sum_{i=0}^N a_i \Gamma_i$, where $a_0 \geq 0$ and $\tilde D_j$ is the strict transform of $D_j$. Moreover, $\tilde \Delta_j$ is a numerical $\tilde \Gamma_j$-curvette on $\tilde X$. Let $X_j$ be the surface $X'$ constructed in Proposition \ref{2compact} with $\delta = \delta_j$. Let $\pi_j: \tilde X \to X_j$ the natural birational morphism, and $\Gamma'_0$ (resp.\ $\Gamma'_j$) be the image of $\tilde \Gamma_0$ (resp.\ $\tilde \Gamma_j$) on $X_j$ (note that $\Gamma'_0 = C_1$ and $\Gamma'_j = C_2$ in the notation of Proposition \ref{2compact}). Set $\Delta'_j := {\pi_j}_*(\tilde \Delta_j) = D'_j + a_0 \Gamma'_0 + a_j \Gamma'_j$, where $D'_j := {\pi_j}_*(\tilde D_j)$.

\begin{proclaim} \label{numerical-claim}
$\Delta'_j \neq 0$. Moreover, $\tilde \Delta_j = \pi_j^*(\Delta'_j)$.
\end{proclaim}

\begin{proof}
Indeed, the only way $\Delta'_j$ can be zero is if $a_0 = a_j = 0$ and $\tilde D_j = 0$. But then we would have $(\tilde \Delta_j, \tilde \Delta_j) = 0$. Since the intersection matrix of $\bigcup_{1\leq i \leq N}\tilde \Gamma_i$ is negative definite, this would imply that $\tilde \Delta_j = 0$, which is a contradiction. This proves the first assertion. The second assertion follows from the fact that the coefficients of exceptional curves in the pullback of a divisor are uniquely determined by the requirement that the intersection number of the pullback with every exceptional curve is zero. 
\end{proof}
W.l.o.g.\ we may assume $\delta_j(f_j) < 0$. Then Proposition \ref{2compact} implies that $(\Gamma'_0,\Gamma'_0) > 0$. Since $(\tilde \Delta_j, \tilde \Gamma_0) = 0$, it follows that 
\begin{align}
(D'_j, \Gamma'_1) = (D'_j, \Gamma'_0) + a_0(\Gamma'_0,\Gamma'_0)  + a_j (\Gamma'_j,\Gamma'_0) = 0. \label{<-equation}
\end{align}
Since $(\Gamma'_0,\Gamma'_0)$ and $(\Gamma'_j,\Gamma'_0)$ are positive, identity \eqref{<-equation} implies that $a_0 = a_j = (D'_j, \Gamma'_0) = 0$. Let $f$ be the polynomial in $\cc[x,y]$ that defines $D'_j \cap \cc^2$. Since $D'_j$ does not intersect $\Gamma'_0$, it follows that $f$ satisfies identity \eqref{delta-factorization} for $\delta:= \delta_j$. Theorem \ref{key-thm} then implies that $f_j$ is a polynomial, as required to complete the proof that \eqref{numerical-assertion} $\im$ \eqref{poly-assertion}.\\

\eqref{poly-assertion} $\im$ \eqref{numerical-assertion}: Fix $j$, $1 \leq j \leq N$. Let $\pi_j: \tilde X \to X_j$ be as above. We claim that there is a numerical $\Gamma'_j$-curvette $D'_j$ on $X_j$. Indeed, if $\delta_j(f_j) \geq 0$, the existence of $D'_j$ follows from Proposition \ref{2compact}. On the other hand, if $f_j$ is a polynomial, then set $D'_j$ to be the closure in $X_j$ of the curve defined by $f_j$ on $\cc^2$. Proposition \ref{puiseuxfinity} then implies that $D'_j$ is a numerical $\Gamma'_j$-curvette, which proves the claim. Now note that $\pi_*(\pi_j^*(D'_j))$ is a numerical $\Gamma_j$-curvette on $X$, which completes the proof of the theorem. 
\end{proof}

\begin{defn} \label{mycurvettinition}
Let $X$ be a normal complete algebraic surface such that 
\begin{enumerate}
\item $X$ contains $U := \cc^2$ as a dense open subset, and 
\item the identity map of $U = \cc^2$ extends to a morphism $X \to X_0 := \pp^2$.
\end{enumerate}
Let $\Gamma_0, \Gamma_1, \ldots, \Gamma_N$ be the irreducible components of $X\setminus U$, where $\Gamma_0$ is the strict transform of the line at infinity on $X_0$. An effective Weil divisor $\Delta$ on $X$ is called a {\em numerical \mytype\ $\Gamma_i$-curvette}, $1 \leq i \leq N$, if 
\begin{enumerate}
\item $\Delta$ is a numerical $\Gamma_i$-curvette,
\item $\supp(\Delta_i)$ contains neither $\Gamma_0$ nor $\Gamma_i$, and
\item $\supp(\Delta_i) \cap U \neq \emptyset$.
\end{enumerate}
We say that $X$ {\em admits a system of numerical \mytype\ curvettes} if there exists a numerical \mytype\ $\Gamma_i$-curvette) for all $i$, $1 \leq i \leq N$.  
\end{defn}

\begin{thm} \label{my-numerical-thm}
Let $X$ be as in Definition \ref{mycurvettinition}. For each $j$, $1 \leq j \leq N$, let $\delta_j$ be the semidegree associated to $\Gamma_j$ and $f_j$ be the {\em last} key form of $\delta_j$. Then the following are equivalent:
\begin{enumerate}[(a)]
\item \label{my-numerical-cartier-assertion} $X$ admits a system of numerical \mytype\ curvettes which are Cartier divisors.
\item \label{my-numerical-assertion} $X$ admits a system of numerical \mytype\ curvettes.
\item \label{my-last-poly-assertion} For each $j$, $1 \leq j \leq N$, the last key form of $\delta_j$ is a polynomial.
\item \label{my-poly-assertion} For each $j$, $1 \leq j \leq N$, all the key forms of $\delta_j$ are polynomials.
\end{enumerate}
\end{thm}

\begin{proof}
\eqref{my-numerical-assertion} $\im$ \eqref{my-last-poly-assertion}: Fix $j$, $1 \leq j \leq N$. Let $X_j$ be as in the proof of Theorem \ref{numerical-thm}, $\pi_j: X \to X_j$ be the natural morphism, $\Gamma'_0 := \pi_j(\Gamma_0)$ and $\Gamma'_j := \pi_j(\Gamma_j)$. Let $\Delta_j := D_j + \sum_{i=0}^N a_i \Gamma_i$ be a numerical \mytype\ $\Gamma_j$-curvette on $X$, where $D_j$ is a non-zero effective Weil divisor such that $\supp(D_j) \not \supset \Gamma_i$ for any $i$, $a_0 = a_j = 0$, and $a_i \geq 0$, $0 \leq i \leq N$. Set $\Delta'_j := {\pi_j}_*(\Delta_j) = D'_j$, where $D'_j := {\pi_j}_*(D_j)$. It follows that $(D'_j, \Gamma'_0) = (D_j,\Gamma_0) = 0$. Let $f$ be the polynomial in $\cc[x,y]$ that defines $D'_j \cap \cc^2$. The same arguments as in the paragraph following the proof of Claim \ref{numerical-claim} then show that the last key form of $\delta_j$ is a polynomial, as required.\\

\eqref{my-last-poly-assertion} $\im$ \eqref{my-numerical-cartier-assertion}: Fix $j$, $1 \leq j \leq N$. Let $\pi_j: X \to X_j$ be as above. Let $D'_j$ be the closure in $X_j$ of the curve defined by the last key form of $\delta_j$ on $U = \cc^2$. Theorem \ref{key-thm} implies that $D'_j \cap \Gamma'_0 = \emptyset$, so that $D'_j$ is a numerical \mytype\ $\Gamma'_j$-curvette $D'_j$ on $X_j$. Moreover, it follows from Proposition \ref{2compact} that for some $n \geq 1$, $nD'_j$ is a Cartier divisor. Then $\pi_j^*(nD'_j))$ is a Cartier divisor on $X$ which is also a numerical \mytype\ $\Gamma_j$-curvette, as required.\\

The equivalence \eqref{my-poly-assertion} $\Leftrightarrow$ \eqref{my-last-poly-assertion} follows from Theorem \ref{key-thm}. Since the implication \eqref{my-numerical-cartier-assertion} $\im$ \eqref{my-numerical-assertion} is obvious, the proof of Theorem \ref{my-numerical-thm} is complete.
\end{proof}

We are now ready to define the central object of this article. 

\begin{defn} \label{polynation}
We denote by $\surface$ the collection of normal complete algebraic surfaces $X$ which satisfy the following properties:
\begin{enumerate}
\item $X$ contains $U := \cc^2$ as an open subset,
\item for each irreducible curve $\Gamma \subseteq X\setminus U$, all key forms associated to (the semidegree corresponding to) $\Gamma$ are polynomials (i.e.\ regular functions on $U$).
\end{enumerate}
$\surfaceplus$ is the subcollection of $\surface$ consisting of $X \in \surface$ which satisfies in addition:
\begin{enumerate}
\setcounter{enumi}{2}
\item For each semidegree $\delta$ on $\cc[U] = \cc[x,y]$ associated to some irreducible curve on $X\setminus U$, $\delta(f) \geq 0$ for all $f \in \cc[x,y]\setminus\{0\}$ (by Theorem \ref{key-thm} this is equivalent to saying that $\delta(f_\delta) \geq 0$, where $f_\delta$ is the {\em last} key form for $\delta$).
\end{enumerate}
\end{defn}

\begin{rem}
If $X$ is a complete rational surface which dominates $\pp^2$ via a birational morphism, then it follows from Theorem \ref{my-numerical-thm} that $X \in \surface$ iff $X$ admits a system of numerical \mytype\ curvettes.
\end{rem}

\begin{rem}
It is not hard to see (e.g.\ follows from Theorem \ref{cox-thm} below) that every surface in $\surface$ is in fact projective.
\end{rem}

Now we describe an important subclass of $\surface$. Recall that in the introduction we considered collections $\scrV$ of pencils $V := \cc \Span{F,H^d}$ where 
\begin{enumerate}
\item $H$ is the equation of a straight line $L$, and
\item $F = F_1^{a_1} \cdots F_k^{a_k}$, where each $F_j$, $1 \leq j \leq k$, defines a curve with one place at infinity, and $\sum_{j=1}^k a_j \deg(F_j) = d$.
\end{enumerate}
For a $V \in \scrV$, recall that $X_V$ denotes the surface resulting from (minimal) resolutions of basepoints of $V$ on $\pp^2$. Also recall that $\scrV_1$ is the subset of $\scrV$ consisting of pencils for which $F$ itself defines an irreducible curve with one place at infinity.

\begin{prop} \label{pencil-surfaces}
\mbox{}
\begin{enumerate}
\item $X_V \in \surface$ for each $V \in \scrV$.
\item If $V \in \scrV_1$, then $X_V \in \surfaceplus$.
\end{enumerate}
\end{prop}

\begin{proof}
Identify $U := \cc^2$ with $\pp^2 \setminus L$. Denote the irreducible components of $X_V\setminus U$ by $\Gamma_0, \ldots, \Gamma_N$, where $\Gamma_0$ is the strict transform of $L$. For each $i$, $0 \leq i \leq N$, let $\delta_i$ be the associated semidegree on $\cc[x,y] := \cc[U]$. At first assume $V \in \scrV_1$, i.e.\ $F$ defines a curve with one place at infinity. In that case dual graph of the irreducible components of  is of the form as in Figure \ref{fig1}.\\

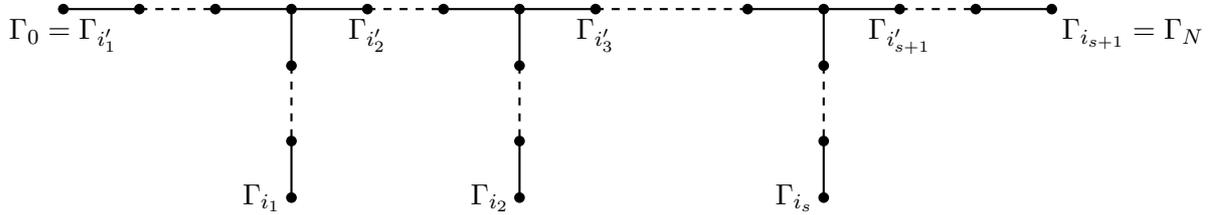
\begin{figure}[htp]
\newcommand{\block}[2]{
 	\pgfmathsetmacro\x{0}
 	\pgfmathsetmacro\y{0}
 	
 	\draw[thick] (\x - \edge,\y) -- (\x,\y);
 	\draw[thick] (\x,\y) -- (\x+\edge,\y);
 	\draw[thick] (\x,\y) -- (\x, \y-\vedge);
 	\draw[thick, dashed] (\x, \y-\vedge) -- (\x, \y-\vedge - \dashedvedge);
 	\draw[thick] (\x, \y-\vedge - \dashedvedge) -- (\x, \y-2*\vedge - \dashedvedge);
 	
 	\fill[black] (\x - \edge, \y) circle (2pt);
 	\fill[black] (\x, \y) circle (2pt);
 	\fill[black] (\x + \edge, \y) circle (2pt);
 	\fill[black] (\x, \y- \vedge) circle (2pt);
 	\fill[black] (\x, \y- \vedge - \dashedvedge) circle (2pt);
 	\fill[black] (\x, \y- 2*\vedge - \dashedvedge) circle (2pt);
 	
 	\draw (\x,\y- 2*\vedge - \dashedvedge)  node (bottom3) [left] {$\Gamma_{i_{#2}}$};
}
\begin{center}
\begin{tikzpicture}
 	\pgfmathsetmacro\edge{1}
 	\pgfmathsetmacro\vedge{.75}
 	\pgfmathsetmacro\dashedvedge{1}
 	
 	\fill[black] (0, 0) circle (2pt);	
 	\draw (0,0)  node (e0) [below] {$\Gamma_0 = \Gamma_{i'_1}$};
 	\draw[thick] (0,0) -- (\edge, 0);
 	\fill[black] (\edge,0) circle (2pt);
 	\draw[thick, dashed] (\edge,0) -- (2*\edge,0);

	\begin{scope}[shift={(3*\edge,0)}]
		\block{2}{1}
	\end{scope}
	\draw (4*\edge,0 )  node [below] {$\Gamma_{i'_2}$};
	\draw[thick, dashed] (4*\edge,0) -- (5*\edge,0);
	
	\begin{scope}[shift={(6*\edge,0)}]
		\block{3}{2}
	\end{scope}	
	\draw (7*\edge,0 )  node [below] {$\Gamma_{i'_3}$};
	\draw[thick, dashed] (7*\edge,0) -- (9*\edge,0);	
	
	\begin{scope}[shift={(10*\edge,0)}]
		\block{s+1}{s}
	\end{scope}
	\draw (11*\edge,0 )  node [below] {$\Gamma_{i'_{s+1}}$};
	
	\draw[thick, dashed] (11*\edge,0) -- (12*\edge,0);	
	\fill[black] (12*\edge,0) circle (2pt);
	\draw[thick] (12*\edge,0) -- (13*\edge,0);	
	\fill[black] (13*\edge,0) circle (2pt);
 	\draw (13*\edge,0)  node (eN) [below right] {$\Gamma_{i_{s+1}} = \Gamma_N$};	
\end{tikzpicture}
\end{center}
\caption{Dual graph of curves at infinity on $X_V$ for $V \in \scrV_1$} \label{fig1}
\end{figure}

It follows from a result of Abhyankar and Moh that there are polynomials $f_1, \ldots, f_s, f_{s+1} = F|_{\cc^2}$, (where $s$ is the number of `vertical segments' in Figure \ref{fig1}), such that for each $j$, $1 \leq j \leq s+1$, the closure $C_j$ in $X_V$ of the curve $\{f_j = 0\} \subseteq \cc^2$ is a numerical $\Gamma_{i_j}$-curvette. Fix $j$, $1 \leq j \leq s+1$, and define
\begin{align*}
[\Gamma_{i'_j}, \Gamma_{i_j}] &:= 
	\begin{cases}
	\text{`L-shaped segment' between (and including) $\Gamma_{i'_j}$ and $\Gamma_{i_j}$} & \text{if}\ 1 \leq j \leq s,\\
	\text{`horizontal segment' between (and including) $\Gamma_{i'_{s+1}}$ and $\Gamma_{i_{s+1}}$} & \text{if}\ j = s+1.
	\end{cases}
\end{align*} 
It follows that for each $l$ such that $\Gamma_l \in [\Gamma_{i'_j}, \Gamma_{i_j}]$, identity \eqref{delta-factorization} is satisfied with $f = f_j$ and $\delta = \delta_l$, and therefore all the key forms of $\delta_l$ are polynomials (via assertion \ref{polynomial-key} of Theorem \ref{key-thm}). Moreover, it is not hard to compute that $\delta_l(f_j) \geq \delta_{i_j}(f_j) \geq 0$, which implies that $\delta_l$ is non-negative on all non-zero polynomials (via assertion \ref{non-negative-key} of Theorem \ref{key-thm}). This proves that $X_V \in \surfaceplus$. \\

Now consider the general case that $F = F_1^{a_1} \cdots F_k^{a_k}$ for $k \geq 1$ and positive integers $a_1, \ldots, a_k$. Then for each $i$, $0 \leq i \leq N$, one of the following (mutually exclusive) cases holds (see e.g.\ \cite[Proof of Theorem 2]{campillo-piltant-lopez-cones}): 
\begin{enumerate}[C{a}se 1.]
\item \label{case-non-contract} $\Gamma_i$ is the strict transforms of some irreducible curve on $X_{V_j}\setminus U$, $1 \leq j \leq k$, where $V_j$ is the pencil generated by $F_j$ and $H^{\deg(F_j)}$. 
\item \label{case-contract} There exists $j$, $1 \leq j \leq k$, such that $-\delta_i$ is a {\em monomial valuation} centered at $O_j$, where $O_j$ is the (unique) point at infinity on the closure $C_j$ in $X_{V_j}$ of the curve $F_j|_U = 0$. Moreover, 
\begin{enumerate}
\item $O_j$ is the point of intersection of $C_j$ and the `last exceptional curve' $\Gamma_{j,N_j}$ on $X_{V_j}$ (i.e.\ $\Gamma_{j,N_j}$ is the exceptional divisor of the last blow up performed in the resolution of the pencil $V_j$).
\item \label{O_j-interior} $O_j$ does {\em not} belong to any irreducible component of $X_{V_j}\setminus U$ other than $\Gamma_{j,N_j}$.
\end{enumerate}
\end{enumerate}
In Case \ref{case-non-contract}, we have already seen that the key forms corresponding to $\delta_i$ are polynomials. So assume Case \ref{case-contract} holds. But then it can be shown that there exists a numerical $\Gamma_i$-curvette on $X$ (see \cite[Proof of Theorem 2]{campillo-piltant-lopez-cones}\footnote{Even though \cite[Theorem 2]{campillo-piltant-lopez-cones} assumes that $\gcd(a_1, \ldots, a_k) = 1$, the proof for existence of systems of numerical curvettes on $X_V$ does not use this assumption.}). Let $\delta_{j,N_j}$ be the semidegree on $\cc[x,y]$ associated to $\Gamma_{j,N_j}$. It follows from the construction of $X_{V_j}$ that $f_j := F_j|_U$ is the last key form of $\delta_{j,N_j}$ and $\delta_{j,N_j}(f_j) = 0$. Property \ref{O_j-interior} then implies that $\delta_i$-value of the last key form of $\delta_i$ is {\em negative}. Theorem \ref{numerical-thm} then implies that the last key form of $\delta_i$ is a polynomial. It follows from Theorem \ref{key-thm} that all key forms of $\delta_i$ are polynomials, as required to complete the proof of the proposition.
\end{proof}

\section{Cox ring of $X \in \surface$} \label{coxection}

\setcounter{thm}{-1}
\begin{notation} \label{set-up}
Throughout the rest of this article we denote by $X$ a surface from the class $\surface$. Let $\Gamma_1, \ldots, \Gamma_N$ be the irreducible components of $X \setminus \cc^2$, and $\delta_i$, $1 \leq i \leq N$, be the order of pole along $\Gamma_i$. Choose coordinates $(x,y)$ on $\cc^2$ such that $\delta_i(x) > 0$ for each $i$. Let $\tilde \phi_i(x,\xi) := \phi_i(x) + \xi x^{r_i}$, $1 \leq i \leq N$, be the generic degree-wise Puiseux series of $\delta_i$. Fix $i$, $1 \leq i \leq N$. Let the {\em formal Puiseux pairs} (Definition \ref{formal-defn}) of $\tilde \phi_i$ be $(q_{i,1}, p_{i,1}), \ldots, (q_{i,l_i+1}, p_{i,l_i+1})$. Recall that the {\em characteristic exponents} (Definition \ref{formal-defn}) of $\tilde \phi_i$ are $r_{i,j} := \frac{q_{i,j}}{p_{i,1}\cdots p_{i,j}}$, $1 \leq j \leq l_i+1$. For each $j$, $0 \leq j \leq l_i$, consider the semidegree $\delta_{ij}$ with generic degree-wise Puiseux series 
\begin{align*}
\tilde \phi_{ij}(x,\xi) := [\phi_i(x)]_{>r_{i,j+1}} + \xi x^{r_{i,j+1}}, 
\end{align*}
where $[\phi_i(x)]_{>r_{i,j+1}}$ denotes the sum of all terms of $\phi_i$ with degree (in $x$) greater than $r_{i,j+1}$.
\end{notation}

Note that $\delta_{il_i} = \delta_i$. For each $j$, $0 \leq j \leq l_i$, let $f_{ij}$ be the last key form of $\delta_{ij}$. The following lemma records some properties of $f_{ij}$'s. It is an immediate corollary of Theorem \ref{key-thm}. 

\begin{lemma} \label{f-ij-lemma}
Fix $i,j$, $1 \leq i \leq N$, $0 \leq j \leq l_i$. Then
\begin{enumerate}[1.]
\let\oldenumi\theenumi
\renewcommand{\theenumi}{$\oldenumi_{(ij)}$}
\item \label{f-ij-1} $f_{ij} = 0$ is a curve with one place at infinity. 
\item \label{f-ij-2} $f_{ij}$ has a degree-wise Puiseux root $\phi_{ij}$ (which is unique up to conjugacy) such that 
\begin{enumerate}
\item Puiseux pairs of $\phi_{ij}$ are $(q_{i,1}, p_{i,1}), \ldots, (q_{i,j}, p_{i,j})$, and
\item $\phi_{ij} = [\phi_i(x)]_{>r_{i,j+1}} +$ terms with degree (in $x$) {\em less} than or equal to $r_{i,j+1}$.  \qed
\end{enumerate}
\end{enumerate}
\end{lemma}

\begin{rem} \label{one-place-f-ij}
Assume that $X = X_V$ for some $V = \cc\Span{F,H^d} \in \scrV$, where $F := F_1^{a_1} \cdots F_k^{a_k}$ such that $f_j := F_j|_{\cc^2}$ defines a curve with one place at infinity for each $j$, $1 \leq j \leq k$. Let $\{g_{i'j'}\}_{j'}$ be the {\em approximate roots} \cite{abhya-moh-tschirnhausen} of $f_{i'}$, $1 \leq i' \leq k$. Then $\{g_{i',j'}\}_{i',j'}$ satisfies the properties of $\{f_{ij}\}_{i,j}$ of Lemma \ref{f-ij-lemma}, i.e.\ for each $i,j$, $1 \leq i \leq N$, $0 \leq j \leq l_i$, there exists $i',j'$ such that properties \ref{f-ij-1} and \ref{f-ij-2} of Lemma \ref{f-ij-lemma} hold with $f_{ij} := g_{i'j'}$.
\end{rem}

\begin{defn} \label{coxinition}
The {\em Cox ring}, or the {\em total coordinate ring}, of $X$ is 
\begin{align*}
\scrR(X) := \dsum_{(d_1, \ldots, d_N) \in \zz^N} \Gamma\left(X, \sheaf\left(\sum d_j \Gamma_j\right)\right)
\end{align*}
For $\vec d := (d_1, \ldots, d_N) \in \zz^N$ and $f \in \Gamma\left(X, \sheaf(\sum d_j \Gamma_j)\right)$, we denote by $(f)_{\vec d}$ the `copy' of $f$ in the $\vec d$'th graded component of $\scrR(X)$. Let $\vec e_j$ be the $j$-th unit vector in $\zz^N$. Moreover, for $f \in \cc(x,y)$ we denote by $\vec\delta(f)$ the element $(\delta_1(f), \ldots, \delta_N(f)) \in \zz^N$.
\end{defn}

\begin{convention}
For each $\vec d := (d_1, \ldots, d_N) \in \zz^N$ and $f \in \Gamma\left(X, \sheaf(\sum d_j \Gamma_j)\right)$, we identify $f$ with $f|_{\cc^2} \in \cc[x,y]$. 
\end{convention}

\begin{thm} \label{cox-thm}
Pick any collection of (not necessarily pairwise distinct) polynomials $f_{ij}$, $1 \leq i \leq N$, $0 \leq j \leq l_i$, which satisfy properties \ref{f-ij-1} and \ref{f-ij-2} of Lemma \ref{f-ij-lemma}. Let $N' := \sum_{i=0}^N (l_i + 1)$. For all $ \vec \gamma := (\gamma_{ij})_{i,j} \in \zz_{\geq 0}^{N'}$ and $\alpha, \beta \in \zz_{\geq 0}$, define $\gabgamma := x^\alpha y^{\beta} \prod_{i,j}f_{ij}^{\gamma_{ij}} \in \cc[x,y]$.
\begin{enumerate}
\item \label{cox-assertion} $\vecdelta{f_{ij}}$'s together with $(x)_{\vec \delta(x)}$, $(y)_{\vec \delta(y)}$, and $(1)_{\vec e_j}$, $1 \leq j \leq N$, generate $\scrR(X)$ as a $\cc$-algebra.
\item \label{line-bundle-assertion} Let $\vec d := (d_1, \ldots, d_N) \in \zz^N$. Let $\scrE_{\vec d}$ be the collection of all $(\alpha, \beta, \vec \gamma) \in \zz_{\geq 0} \times \zz_{\geq 0} \times \zz_{\geq 0}^{N'}$ such that 
\begin{align}
\alpha \delta_k(x) + \beta\delta_k(y) + \sum_{i,j}\gamma_{ij}\delta_k(f_{ij}) \leq d_k \tag{$*_k$} \label{h0-k}
\end{align}
is satisfied for all $k$, $1 \leq k \leq N$. Then $\Gamma\left(X, \sheaf(\sum d_j \Gamma_j)\right)$ is a generated as a vector space over $\cc$ by $\{\gabgamma: (\alpha, \beta, \vec \gamma)\in \ed\} $.
\item \label{basis-assertion} Let $\vec d := (d_1, \ldots, d_N) \in \zz^N$. Define
\begin{align*}
\scrE'_{\vec d} 
	&:= \{(a,b) \in \zz_{\geq 0}^2: b = \beta + \sum_{i,j}\gamma_{ij}\deg_y(f_{ij}) = \deg_y(g_{a,\beta,\vec\gamma})\\ 
	& \qquad \quad \text{for some $\beta \in \zz_{\geq 0}$ and $\vec\gamma = (\gamma_{ij})_{i,j} \in \zz_{\geq 0}^{N'}$ such that}\ (a, \beta, \vec\gamma)\in \scrE_{\vec d}\}. 	
\end{align*} 
Then $\dim \Gamma\left(X, \sheaf(\sum d_j \Gamma_j)\right) = |\scrE'_{\vec d}|$. For each $(a,b) \in \scrE'_{\vec d}$, pick any $\beta \in \zz_{\geq 0}$ and $\vec\gamma \in \zz_{\geq 0}^{N'}$ such that $(a,\beta,\vec \gamma) \in \ed$ and $b = \deg_y (g_{a,\beta,\vec \gamma})$, and define $g_{a,b} := g_{a,\beta,\vec \gamma}$. Then $\{g_{a,b}: (a,b) \in \scrE'_{\vec d}\}$ is a basis of $\Gamma\left(X, \sheaf(\sum d_j \Gamma_j)\right)$.
\end{enumerate}
\end{thm}

\begin{proof}
Note that for all $(\alpha, \beta, \vec \gamma) \in \scrE_{\vec d}$, we have
\begin{align*}
\vec \delta\left(\gabgamma\right) 
	&= \vec \delta\left(x^\alpha y^{\beta} \prod_{i,j}f_{ij}^{\gamma_{ij}}\right)
	= \alpha \vec \delta(f) + \beta \vec\delta(y) + \sum_{i,j}\vec \delta(f_{ij}) \leq \vec d
\end{align*}
(where ``$\leq$'' denotes coordinate-wise inequality). In particular,
\begin{align}
(\alpha, \beta, \vec \gamma) \in \scrE_{\vec d} \iff \gabgamma \in \Gamma\left(X, \sheaf(\sum d_j \Gamma_j)\right). \label{Ed-interpretation}
\end{align}
This implies that assertion \ref{cox-assertion} follows from assertion \ref{line-bundle-assertion}. Moreover, assertion \ref{basis-assertion} also is a corollary of assertion \ref{line-bundle-assertion}. Therefore it suffices to prove assertion \ref{line-bundle-assertion}.\\ 

Pick $\vec d \in \zz^N$ and $h \in \Gamma\left(X, \sheaf(\sum d_j \Gamma_j)\right)$. We will show by induction on $\deg_y(h)$ that $h$ belongs to the vector space generated by $\scrG_{\vec d} := \{\gabgamma: (\alpha, \beta, \vec \gamma)\in \ed\} $. Indeed, if $\deg_y(h) = 0$, then it is clear. So assume $\deg_y(h) \geq 1$, and consider the degree-wise Puiseux factorization of $h$:
\begin{align*}
h = cx^m \prod_{i=1}^M \prod_{\parbox{1.75cm}{\scriptsize{$\psi_{ij}$ is a con\-ju\-ga\-te of $\psi_i$}}}\mkern-27mu \left(y - \psi_{ij}(x)\right).
\end{align*}
Fix $i$, $1 \leq i \leq M$. For each $j,j'$, $1 \leq j \leq N$, $1 \leq j' \leq d^*_i$ (where $d^*_i$ is the number of conjugates of $\psi_i$), define 
\begin{align*}
s_{ij'j} &:=\deg_x(\psi_{ij'}(x) - \tilde \phi_{j}(x,\xi)) = \deg_x(\psi_{ij'}(x) - \phi_{j}(x) - \xi x^{r_j}),\\
s_i &:= \min\{s_{ij'j}\}_{j',j}.
\end{align*}
Pick $j'_i$, $j_i$ such that $s_i = s_{ij'_ij_i}$. We may assume w.l.o.g.\ that $j'_i = 1$. Let 
\begin{align*}
\psi^*_i &:= [\psi_{i1}(x)]_{> s_i} = [\phi_{j_i}(x)]_{> s_i}.
\end{align*}
For each $k$, $0 \leq k \leq l_{j_i}$, let $\phi_{j_ik}$ be a degree-wise Puiseux root of $f_{j_ik}$ (recall that $f_{j_ik}$'s satisfy properties \ref{f-ij-1} and \ref{f-ij-2} of Lemma \ref{f-ij-lemma}). Since $s_i \geq r_{j_i}$, it follows that if $\psi^*_i \neq 0$, then there is a unique $k_i$, $0 \leq k_i \leq l_{j_i}$, such that
\begin{enumerate}[(a)]
\item \label{phi-j-i-k-i-1} $\phi_{j_ik_i}$ has the same Puiseux pairs as $\psi^*_i$, and
\item \label{phi-j-i-k-i-2} a conjugate of $\phi_{j_ik_i}$ is of the form $\psi^*_i +$ terms with degree less than or equal to $s_i$.
\end{enumerate}
Define
\begin{align*}
h_i := \prod_{\parbox{1.75cm}{\scriptsize{$\psi_{ik}$ is a con\-ju\-ga\-te of $\psi_i$}}}\mkern-27mu \left(y - \psi_{ik}(x)\right) 	
	\in \dpsxc[y].
\end{align*}
Note that $h = \prod_{i=1}^M h_i$ and $\deg_y(h_i) = $ number of conjugates of $\psi_i$. Let
\begin{align*}
m_i := \begin{cases}
	\deg_y(h_i)& \text{if}\ \psi^*_i = 0,\\
	\frac{\text{number of conjugates of $\psi_i$}}{\text{number of conjugates of $\psi^*_i$}} & \text{otherwise.}
	\end{cases} 
	\qquad
h^*_{i} := 
	\begin{cases}
	y^{m_i} & \text{if}\ \psi^*_i = 0,\\
	f_{j_ik_i}^{m_i} & \text{otherwise.}
	\end{cases}
\end{align*}
Note that $m_i$ is a positive integer and $h^*_{i} \in \cc[x,y]$. Define 
$$h^* := cx^m \prod_{i=1}^M h^*_i \in \cc[x,y].$$

\begin{proclaim}\label{delta-j-f*-f}
$\delta_j(h^*) \leq \delta_j(h)$ for each $j$, $1 \leq j \leq N$.
\end{proclaim}

\begin{proof}
Fix $i,j$, $1 \leq i \leq M$, $1 \leq j \leq N$. It suffices to show that $\delta_j(h^*_i) \leq \delta_j(h_i)$. As in \eqref{puiseux-degree}, 
\begin{align*}
\delta_j(h_i) 
	&= \delta_j(x)\deg_x\left(h_i(x,y)|_{y = \tilde \phi_j(x,\xi)}\right)\\
	&= \delta_j(x)	\deg_x\left(\prod_{k =1}^{d^*_i} (\tilde \phi_j(x,\xi) - \psi_{ik}(x))\right) = \delta_j(x)\sum_{k=1}^{d^*_i} s_{ikj},
\end{align*}
where $d^*_i := \deg_y(h_i) = \deg_y(h^*_i)$. If $\psi^*_i = 0$, then $s_{ikj} \geq \deg_x(\tilde \phi_j(x,\xi))$ for all $k$, which implies that 
\begin{align*}
\delta_j(h_i) \geq \delta_j(x) d^*_i \deg_x(\tilde \phi_j(x,\xi)) = d^*_i\delta_j(y) = \delta_j(y^{d^*_i}) = \delta_j(h^*_i),
\end{align*}
as required to prove the claim. So assume $\psi^*_i \neq 0$. Define 
\begin{align*}
s_{ij} &:= \min\{s_{ij'j}\}_{j'}.\\
s'_{ij} &:=\min\{\deg_x(\tilde \phi_{j}(x,\xi) - \phi(x)): \phi\ \text{is a conjugate of}\ \phi_{j_ik_i} \}. 
\end{align*}
Claim \ref{claim-sij-s'ij} and Lemma \ref{delta-comparison} below then imply that $\delta_j(h_i) \geq \delta_j(h^*_i)$, as required to prove Claim \ref{delta-j-f*-f}.
\end{proof}

\begin{proclaim} \label{claim-sij-s'ij}
$s_{ij} \geq s'_{ij}$.
\end{proclaim}

\begin{proof}
Indeed, replacing $\phi_j$ by one of its conjugates if necessary, we may assume that $s_{ij} = s_{i1j}$. Similarly, property \eqref{phi-j-i-k-i-2} of $\phi_{j_ik_i}$ implies that replacing $\phi_{j_ik_i}$ by one of its conjugates if necessary, we may assume that $\phi_{j_ik_i}$ is of the form $[\psi_{i1}(x)]_{> s_i} +$ terms with degree less than or equal to $s_i$. Since $s_{ij} \geq s_i$, this implies that $\deg_x(\tilde \phi_{j}(x,\xi) - \phi_{j_ik_i}(x)) \leq s_{ij}$, as required.
%
%
%
\end{proof}
Claim \ref{delta-j-f*-f} implies that $h^* \in \Gamma\left(X, \sheaf(\sum d_j\Gamma_j)\right)$, so that $h - h^*$ is also an element of $\Gamma\left(X, \sheaf(\sum d_j \Gamma_j)\right)$. Since $\deg_y(h-h^*) < \deg_y(h)$, the inductive hypothesis implies that $h-h^*$ is in the vector space generated by $\scrG_{\vec d}$. Now note that by construction $h^* = c\gabgamma$ for some $\alpha, \beta, \vec \gamma$. Claim \ref{delta-j-f*-f} and \eqref{Ed-interpretation} then imply that $\gabgamma \in \scrG_{\vec d}$, and complete the proof of the theorem.
\end{proof}

\section{Global Enriques and Zariski semigroups of divisors on $X \in \surface$} \label{linection}
Global Enriques and Zariski semigroups of line bundles on surfaces of the form $X_V$, where $V$ is a pencil at infinity, were introduced in \cite{campillo-piltant-lopez-cones-surfaces}. In this section we adapt these notions for surfaces in $\surface$ and then compute these semigroups. We continue with the set up introduced in Notation \ref{set-up}.

\begin{defn}
Recall that the divisor class group $\cl(X)$ (of Weil divisors on $X$ modulo linear equivalence) is isomorphic to $\zz^N$ and generated by $\Gamma_1, \ldots, \Gamma_N$. The {\em global Enriques semigroup} $\pst(X)$ of $X$ consists of all $D \in \cl(X)$ such that $D$ is linearly equivalent to an effective Weil divisor whose support does not contain any of the $\Gamma_j$'s. The {\em global Zariski semigroup}, or the {\em characteristic semigroup} of $X$ is the semigroup $\tilde P(X)$ generated by all divisors in $\cl(X)$ which are base point free. The {\em Zariski semigroup at infinity} $\tilde P_\infty$ is generated by all divisiors in $\cl(X)$ which have no base point `at infinity', i.e.\ the base locus consists of finitely many points on $\cc^2$. 
\end{defn}

\begin{rem} \label{positively-zariski-infinite}
If $X \in \surfaceplus$ (e.g.\ if $X = X_V$ for some $V \in \scrV_1$) then $\tilde P_\infty = \tilde P$. 
\end{rem}

For $\vec d \in \zz^N$, we denote by $\Gammad$ the divisor $\sum_j d_j\Gamma_j \in \cl(X)$. Theorem \ref{cox-thm} gives an immediate characterization of those $\vec d$ for which $\Gammad \in \pstx$. 

\begin{cor} \label{enricorollary}
Let $f_{ij}$, $1 \leq i \leq N$, $0 \leq j \leq l_i$, be as in Theorem \ref{cox-thm}. For each $\vec d \in \zz^N$ and each $j$, $1 \leq j \leq N$, define: 
\begin{align*}
M_{\vec d,j} 
	&:= \max\left\{\delta_j(\gabgamma): (\alpha,\beta,\vec \gamma) \in \ed\right\} \\
	&= \max\left\{\alpha \delta_j(x) + \beta\delta_j(y) + \sum_{i,j}\gamma_{ij}\delta_j(f_{ij}): (\alpha,\beta,\vec \gamma) \in \ed\right\}.
\end{align*}
Then $\Gammad \in \pstx$ iff $d_j = M_{\vec d,j}$ for all $j$, $1 \leq j \leq N$. \qed
\end{cor}


\begin{defn} \label{ndefn}
Fix $i$, $1 \leq j \leq N$. Let $\tilde \delta_i := \delta_{i}/\delta_{i}(x)$ be the `normalized' version of $\delta_i$. Recall that the {\em formal Puiseux pairs} (Definition \ref{formal-defn}) of $\tilde \phi_i$ are $(q_{i,1}, p_{i,1}), \ldots, (q_{i,l_i+1}, p_{i,l_i+1})$. Define:
\begin{align*}
p_{i} &:= p_{i,1} \cdots p_{i,l_i},\\
\tilde p_{i} &:= \delta_i(x) = p_{i}p_{i,l_i +1} = p_{i,1} \cdots p_{i, l_i + 1}.
\end{align*}
Let $f_{ij}$, $1 \leq i \leq N$, $0 \leq j \leq l_i$, as in Theorem \ref{cox-thm}. For each $i,j$, let $\phi_{ij}$ be a degree-wise Puiseux root of $f_{ij}$. Fix $i$, $1 \leq i \leq N$, and define $\scrN_i := \{i': \tilde \delta_{i'} \leq \tilde \delta_{i}\}$. Recall (Property \ref{f-ij-2} implies) that $f_{il_i}$ has a generic degree-wise Puiseux root $\phi_{il_i}(x)$ such that 
\begin{align*}
\phi_{il_i} = [\phi_i(x)]_{>r_{i}} + \ \text{terms with degree (in $x$) {\em less} than or equal to}\ r_{i}
\end{align*}
For each $i' \in \scrN_i$, (replacing $\phi_{i'}$ by a conjugate if necessary we may assume that) there exists $c_{ii'} \in \cc$ such that 
\begin{align}
\phi_{i'l_{i'}} = [\phi_i(x)]_{> r_i} + c_{ii'}x^{r_i} + \ldt \label{c-ii'}
\end{align}
Note that $c_{ii'}$ is unique only up to a multiplication by a $p_{i,l_i+1}$-th root of unity. For each $c \in \cc$, define
\begin{align*}
\scrN_{i,c} &:= \left\{i' \in \scrN_i\setminus\{i\}: c_{ii'}^{p_{i,l_i+1}} = c \right\}.
\end{align*}
Then $\scrN_{i}= \bigcup_{c \in \cc}\scrN_{i,c} \cup \{i\}$. For convenience we record the following lemma which is a direct consequence of Lemma \ref{f-ij-lemma}:
\end{defn}

\begin{lemma} \label{c-ii-lemma}
If $c_{ii} \neq 0$, then $p_{i,l_i+1} = 1$. In particular, identity \eqref{c-ii'} for $i := i'$ uniquely determines $c_{ii}$. \qed
\end{lemma}

\begin{reminition} \label{lc-remark}
For each $f \in \cc[x,y]\setminus\{0\}$, and each $i$, $1 \leq i \leq N$, we write $\lc_{i}(f)$ for the coefficient of $x^{\tilde \delta_{i}(f)}$ in $f|_{y = \tilde \phi_{i}(x,\xi)}$. Note that $\lc_i(f)$ can be factored as $\xi^r h(\xi^{p_{i,l_i+1}})$, where 
\begin{enumerate}
\item $r$ is the number of degree-wise Puiseux roots $\psi(x)$ of $f$ such that $\deg_x(\phi_i - \psi(x)) < r_i$,
\item $h$ is a polynomial in one variable,
\item $p_{i,l_i+1}\deg(h)$ is the number of degree-wise Puiseux roots $\psi(x)$ of $f$ such that $\deg_x(\phi_i - \psi(x)) = r_i$.
\end{enumerate}
\end{reminition}

\begin{thm} \label{zariski-thm}
For $\vec a := (a_1, \ldots, a_N) \in \zz_{\geq 0}^N$, define
\begin{align*}
d_i(\vec a) &:= \sum_{j=1}^N a_j\omega_{ij},\ 1 \leq i \leq N,\ \text{where}\\
\omega_{ij} &:= 
	\begin{cases}
		p_{j,l_j+1}\delta_i(f_{jl_j}) & \text{if}\ i \not\in \scrN_j\\
		p_{j,l_j+1}\tilde p_i \tilde \delta_j(f_{jl_j}) & \text{if}\ i \in \scrN_j.
	\end{cases}\\
\vec d_{\vec a} &:= (d_1(\vec a), \ldots, d_N(\vec a)), \\
m_{i}(\vec a) &:= \sum_{i' \in \scrN_{i}} \frac{\tilde p_{i'}a_{i'}}{p_{i}},\ 1 \leq i \leq N, \\
m_{i,c}(\vec a) &:= \sum_{i' \in \scrN_{i,c}} \frac{\tilde p_{i'}a_{i'}}{p_{i}},\ c \in \cc,\ 1 \leq i \leq N.
\end{align*} 
Let $f_{ij}$'s be as in Theorem \ref{cox-thm}. Then for each $\vec d \in \zz^N$, the following are equivalent:
\begin{enumerate}[(a)]
\item \label{bpf} $\sheaf(\sum d_j \Gamma_j)$ has no base point at infinity.
\item \label{bpf-AB} $\vec d = \vec d_{\vec a}$ for some $\vec a \in \zz_{\geq 0}^N$ such that for each $i$, $1 \leq i \leq N$, 
\begin{align}
d_{i}(\vec a) &= \max\left\{\delta_{i}\left(g_{\alpha\beta\vec\gamma}\right): (\alpha,\beta,\vec\gamma) \in \scrE_{\vec d_{\vec a}}\right\}, \label{not-sure-property} \\
m_{i}(\vec a) &= \max\left\{\deg_\xi\left(\lc_{i}\left(g_{\alpha\beta\vec\gamma}\right)\right): (\alpha,\beta,\vec\gamma) \in \scrE_{\vec d_{\vec a}},\ \delta_{i}(g_{\alpha\beta\vec\gamma})= d_{i}(\vec a) \right\}, \label{max-xi} \\
m_{i,c_{ii}}(\vec a) &= \min\left\{\ord_{\xi-c_{ii}}\left(\lc_{i}\left(g_{\alpha\beta\vec\gamma}\right)\right): (\alpha,\beta,\vec\gamma) \in \scrE_{\vec d_{\vec a}},\ \delta_{i}(g_{\alpha\beta\vec\gamma})= d_{i}(\vec a) \right\}, \label{min-xi} \\
a_ip_{i,l_i+1} + m_{i,c_{ii}}(\vec a) &= \max\left\{\ord_{\xi-c_{ii}}\left(\lc_{i}\left(g_{\alpha\beta\vec\gamma}\right)\right): (\alpha,\beta,\vec\gamma) \in \scrE_{\vec d_{\vec a}},\ \delta_{i}(g_{\alpha\beta\vec\gamma})= d_{i}(\vec a) \right\}. \label{max-order-xi} 
\end{align} 
\end{enumerate}
\end{thm}

\begin{proof}
\eqref{bpf} $\im$ \eqref{bpf-AB}: Assume $\sheaf(\sum d_j \Gamma_j)$ has no base point at infinity. Let $C$ be the divisor corresponding to a generic global section of $\sheaf(\sum d_j \Gamma_j)$. By Bertini's theorem $C$ is transversal to the curve at infinity, and for each $j$,
\begin{enumerate}
\item $C$ intersects each $C_j$ (transversally) at $a_j$ distinct points for some $a_j \geq 0$,
\item $C \cap C_i \cap C_j = \emptyset$ for each $i \neq j$,
\item $C$ does not pass through the `zero point' of $C_j$ for any $j$.
\end{enumerate}
(note that $a_1, \ldots, a_N$ are independent of $C$). Pick $h \in \cc[x,y]$ which defines $C \cap \cc^2$. Then $h$ has a degree-wise Puiseux factorization of the form 
\begin{align*}
h = c\prod_{i=1}^N \prod_{j=1}^{a_i} \prod_{k=1}^{\tilde p_{i}} \left(y - \psi_{ijk}(x)\right)
\end{align*}
where $c \in \cc^*$, and $\psi_{ijk}$'s are conjugates of some $\psi_{ij}$ of the form $\phi_i(x) + \xi_{ij} x^{r_i} + \lot$ for $\xi_{ij} \in \cc^*$ which are `base point free' in the sense that
\begin{align}
\parbox{.72\textwidth}{for every finite set $S \subseteq \cc^*$, if $h$ is generic enough, then $\{\xi_{ij}\}_{j} \cap S = \emptyset$.} \label{bpf-roots}
\end{align}
A straightforward computation then shows that for each $i,i',j$, $1 \leq i,i' \leq N$, $1 \leq j \leq a_i$,
\begin{align}
& \delta_{i'}\left(\prod_{k=1}^{\tilde p_{i}} (y - \psi_{ijk}(x))\right) = \omega_{i'i},\ \text{so that} \notag \\
& \delta_{i'}(h) = \sum_{i=1}^N a_i\omega_{i'i} = d_{i'}(\vec a). \label{i'j'-bpf-1}
\end{align}
It follows that $\vec d = \vec d_{\vec a}$. The first assertion of the following claim is a straightforward implication of the `base point freeness' of $\xi_{ij}$'s, and the second assertion follows from the genericness of $h$.

\begin{proclaim} \label{bpf-leading-claim}
Fix $i'$, $1 \leq i' \leq N$. Let 
\begin{align}
\eta_{i'}(\xi) := \prod_{i \in \scrN_{i'}\setminus\{i'\}} (\xi^{p_{i',l_{i'}+1}} - c_{i'i}^{p_{i',l_{i'}+1}})^{(\tilde p_ia_i)/\tilde p_{i'}}. \label{eta}
\end{align}
Then 
\begin{enumerate}
\item \label{bpf-generic} $\lc_{i'}(h)  = c'\eta_{i'}(\xi) \prod_{j=1}^{a_{i'}}\left(\xi^{p_{i',l_{i'}+1}} - \xi_{i'j}^{p_{i',l_{i'}+1}} \right)$ for some $c' \in \cc^*$.
\item \label{bpf-all} Let $h' \in \Gamma(X, \sheaf(\sum_j d_j\Gamma_j))$ such that $\delta_{i'}(h') = d_{i'}$. Then $\lc_{i'}(h') = \eta_{i'}(\xi)\mu_{h'}(\xi)$ for some $\mu_{h'} \in \cc[\xi]$.\qed
\end{enumerate}
\end{proclaim}

Theorem \ref{cox-thm}, identity \eqref{i'j'-bpf-1}, Claim \ref{bpf-leading-claim} and Lemma \ref{c-ii-lemma} together imply that $\vec a$ satisfies \eqref{not-sure-property}, \eqref{max-xi} and \eqref{min-xi}.\\

Now we show that $\vec a$ satisfies \eqref{max-order-xi}. Recall the procedure from the proof of Theorem \ref{cox-thm} to write $h$ as a linear combination of the $\gabgamma$. Set $h_0 := h$, and inductively write $h_{k+1} = h_k - h_k^*$, where $h_k^*$ is computed from $h_k$ in the same way as $h^*$ is computed from $h$ in the proof of Theorem \ref{cox-thm}. 

\begin{proclaim} \label{ki'-claim}
Fix $i'$, $1 \leq i' \leq N$. Let $k_{i'} \geq 0$ be the smallest integer such that $\delta_{i'}(h_{k_{i'}}^*) = \delta_{i'}(h)$. Then for each $k$, $0 \leq k \leq k_{i'}$, the degree-wise Puiseux factorization of $h_k$ is of the form
\begin{align}
h_k = g_k(x,y)\prod_{i \in \scrN_{i'}} \prod_{j=1}^{a_i} \prod_{k=1}^{\tilde p_{i}} \left(y - \psi_{ijk}(x)\right), \label{hk-factorization}
\end{align}
for some $g_k(x,y) \in \dpsxc[y]$.
\end{proclaim}

\begin{proof}
We prove the claim by induction on $k$. It is clearly true if $k = 0$. Now pick $k$, $0 \leq k < k_{i'}$, such that \eqref{hk-factorization} holds for $k$. We will show that \eqref{hk-factorization} holds for $k+1$. Write the degree-wise Puiseux factorization of $h_k$ as 
\begin{align*}
h_k = cx^m \prod_{k'=1}^K \prod_{\parbox{1.75cm}{\scriptsize{$\psi_{k'l'}$ is a con\-ju\-ga\-te of $\psi_{k'}$}}}\mkern-27mu \left(y - \theta_{k'l'}(x)\right)
\end{align*}
Recall that $h_k^*$ is formed from \eqref{hk-factorization} by replacing each $h_{kk'} := \prod_{l'} (y - \theta_{k'l'}(x))$ by $h^*_{kk'} := f_{i_{k'}j_{k'}}^{n_{k'}}$ for certain $i_{k'},j_{k'}, n_{k'}$. Since $k < k_{i'}$, we have $\delta_{i'}(h_k) > \delta_{i'}(h_k^*)$. This implies that there exists $k'$ such that $ \delta_{i'}(h_{kk'}) > \delta_{i'}(h^*_{kk'})$. It then follows (e.g.\ due to Lemma \ref{delta-comparison}) that
\begin{align*}
s_{i'k'} := \min_{l'}\{\deg_x(\phi_{i'}(x) + \xi x^{r_{i'}} - \theta_{k'l'}(x)\} > \min_l \{\deg_x(\phi_{i'}(x) + \xi x^{r_{i'}} - \phi_{i_{k'}j_{k'}l}(x)\} \geq r_{i'},
\end{align*}
where $\phi_{i_{k'}j_{k'}l}$'s on the right hand side runs over all conjugates of $\phi_{i_{k'}j_{k'}}$. Pick $l'_0$ (resp.\ $l_0$) for which the minimum is achieved on the left (resp.\ right) hand side of the preceding inequality. Then 
\begin{align*}
\psi_{k'l'_0}(x) &= \phi_{i'}(x) + bx^{s_{i'k'}} + \ldt\\
\phi_{i_{k'}j_{k'}l_0}(x) &= \phi_{i'}(x) + \text{terms with degree less than}\ {s_{i'k'}}.
\end{align*}
Now fix $i \in \scrN_{i'}$. Then $\phi_i$ has a conjugate which agrees with $\phi_{i'}$ up to degree $r_{i'}$, so that 
\begin{align*}
\min_l \{\deg_x(\phi_{i}(x) + \xi x^{r_{i}} - \phi_{i_{k'}j_{k'}l}(x)\} < s_{i'k'} = \min_{l'}\{\deg_x(\phi_{i}(x) + \xi x^{r_{i}} - \theta_{k'l'}(x)\}.
\end{align*}
Lemma \ref{delta-comparison} then implies that $ \delta_{i}(h_{kk'}) > \delta_{i}(h^*_{kk'})$, so that $ \delta_{i}(h_k) > \delta_{i}(h^*_k)$. By induction hypothesis the (closure in $X$ of the) curve $h_k = 0$ intersects $\Gamma_i$ transversally at $a_i$ points. It follows that the curve defined by $h_{k+1} = h_k - h_k^*$ also intersects $\Gamma_i$ transversally at each of those $a_i$ points. This implies that \eqref{hk-factorization} holds for $k+1$, as required. 
\end{proof}

Fix $i,j$, $i \in \scrN_{i'} \setminus \{i'\}$, $1 \leq j \leq a_i$. Claim \ref{ki'-claim} implies that $h_{ij} := \prod_{k=1}^{\tilde p_{i}} \left(y - \psi_{ijk}(x)\right)$ divides $h_{k_{i'}}$ in $\dpsxc[y]$. Recall that $\psi_{ijk}$'s are conjugates of $\psi_{ij}$ and $\psi_{ij} = \phi_i (x) + \xi_{ij} x^{r_i} + \lot$. The arguments in the proof of Theorem \ref{cox-thm} show that the factor $h_{ij}^*$ of $h_{k_{i'}}^*$ corresponding to $h_{ij}$ can be defined to be $f_{il_{i}}^{p_{i,l_{i}+1}}$. It follows that $\lc_{i'}(h_{k_{i'}}^*)$ has a factor of the form
\begin{align*}
\sigma(\xi) &:= \prod_{i\in \scrN_{i'}}\prod_{j=1}^{a_i}\lc_{i'}(h_{ij}^*) = \prod_{i \in \scrN_{i'}} (\xi^{p_{i',l_{i'}+1}} - c_{i'i}^{p_{i',l_{i'}+1}})^{(\tilde p_ia_i)/\tilde p_{i'}} = \eta_{i'}(\xi)  (\xi^{p_{i',l_{i'}+1}} - c_{i'i'}^{p_{i',l_{i'}+1}})^{a_{i'}}.
\end{align*}
Since $\deg(\sigma) = m_{i'}(\vec a)$, identity \eqref{max-xi} implies that $\lc_{i'}(h_{k_{i'}}^*) = c_{i'}\sigma(\xi)$ for some $c_{i'} \in \cc^*$, and consequently we get, using Lemma \ref{c-ii-lemma}, that 
\begin{align*}
\ord_{\xi-c_{i'i'}}\left( \lc_{i'}(h_{k_{i'}}^*) \right) = \ord_{\xi-c_{i'i'}}(\sigma(\xi)) = m_{i',c_{i'i'}}(\vec a) + a_{i'}p_{i',l_{i'}+1}. 
\end{align*} 
On the other hand, assertion \ref{bpf-all} of Claim \ref{bpf-leading-claim}, coupled with identity \eqref{max-xi} and Lemma \ref{c-ii-lemma} imply that $\ord_{\xi-c_{i'i'}}(h') \leq m_{i',c_{i'i'}}(\vec a) + a_{i'}p_{i',l_{i'}+1}$ for all $h' \in \Gamma(X, \sheaf(\sum_j d_j\Gamma_j))$ such that $\delta_{i'}(h') = d_{i'}$. The proves \eqref{max-order-xi}, and finishes the proof that \eqref{bpf} $\im$ \eqref{bpf-AB}.\\

\eqref{bpf-AB} $\im$ \eqref{bpf}: Pick $\vec a$ that satisfies \eqref{not-sure-property}, \eqref{max-xi}, \eqref{min-xi}, \eqref{max-order-xi}, and set $\vec d := \vec d_{\vec a}$. At first  the following claim: 

\begin{proclaim}\label{bpfactorization-claim} 
Fix $i$, $1 \leq i \leq N$. Let $\eta_i(\xi)$ be as in \eqref{eta}. A generic $h \in \Gamma(X, \sheaf(\sum d_j \Gamma_j))$ satisfies $\lc_{i}(h)  = c\eta_{i}(\xi) \prod_{j=1}^{a_{i}}\left(\xi^{p_{i,l_{i}+1}} - \xi_{ij}^{p_{i,l_{i}+1}} \right)$ for some $c \in \cc^*$ and $\{\xi_{ij}\}_{i,j} \subseteq \cc^*$ that satisfies the `base point free condition' \eqref{bpf-roots}. In other words, assertion \ref{bpf-generic} of Claim \ref{bpf-leading-claim} is true for $i' := i$.
\end{proclaim}

\begin{proof}
We prove the claim by induction on $|\scrN_{i}|$. Indeed, if $|\scrN_{i}| = 1$, then $\eta_i(\xi) = 1$, $m_{i,c_{ii}}(\vec a) = 0$, and $m_{i}(\vec a) = p_{i,l_i + 1}a_i$. By assumption there exists $h_1, h_2 \in \Gamma(X, \sheaf(\sum d_j \Gamma_j))$ such that 
\begin{enumerate}
\item $\delta_i(h_1) = \delta_i(h_2) = d_i$,
\item $\lc_{i}(h_1) = c_1(\xi-c_{ii})^{m_i(\vec a)},\ c_1 \in \cc^*$, and
\item $\ord_{\xi-c_{ii}}(\lc_{i}(h_2)) = 0$.
\end{enumerate}
It follows that for a generic $\lambda \in \cc^*$, the roots $\xi_{ij}$ of $\lc_{i}(h_1 + \lambda  h_2)$ satisfies the base point free condition \eqref{bpf-roots}. Remark-Definition \ref{lc-remark} then implies that Claim \ref{bpfactorization-claim} holds for $i$.\\

Now pick $i'$ such that Claim \ref{bpfactorization-claim} holds for all $i$ with $|\scrN_{i}|< |\scrN_{i'}|$. 

\begin{proclaim} \label{star-claim'}
Let $h' \in \Gamma(X, \sheaf(\sum_j d_j\Gamma_j))$ such that $\delta_{i'}(h') = d_{i'}$. Then $\lc_{i'}(h') = \eta_{i'}(\xi)\mu_{h'}(\xi)$ for some $\mu_{h'} \in \cc[\xi]$. In other words, assertion \ref{bpf-all} of Claim \ref{bpf-leading-claim} holds for $i'$.
\end{proclaim}

\begin{proof}
Pick a {\em generic} $h' \in \Gamma(X, \sheaf(\sum_j d_j\Gamma_j))$ such that $\delta_{i'}(h') = d_{i'}$. It suffices to show that Claim \ref{star-claim'} holds for $h'$. Fix $i \in \scrN_{i'}\setminus\{i'\}$. Property \eqref{not-sure-property} implies that $\delta_i(h) = d_i$. Let $\zeta$ be a primitive $p_{i,l_i+1}$-th root of unity. The inductive hypothesis implies that $h'$ has $a_ip_{i,l_i+1}$ degree-wise Puiseux roots of the form 
$$\psi_{ijk}(x) = \phi_i(x) + \zeta^k \xi_{ij}x^{r_i} + \ldt $$
for $1 \leq k \leq p_{i,l_i+1}$, $1 \leq j \leq a_i$, and $\xi_{ij}$'s that satisfy the `base point free condition' \eqref{bpf-roots}. For each such root, there are $p_{i}/p_{i'}$ conjugates which are of the form 
$$\phi_{i'}(x) + c_{i'i}x^{r_{i'}} + \ldt$$
where $c_{i'i}$ is as in Definition \ref{ndefn}. It follows that each $i \in \scrN_{i'}\setminus\{i'\}$ contributes a factor of $\eta_{i'i}(\xi) := (\xi^{p_{i',l_{i'}+1}} - c_{i'i}^{p_{i',l_{i'}+1}})^{a_i\tilde p_i/\tilde p_{i'}}$ to $\lc_{i'}(h')$. Consequently $\lc_{i'}(h')$ has a factor of the form $\prod_{i \in \scrN_{i'}\setminus\{i'\}} \eta_{i'i}(\xi) = \eta_{i'}(\xi)$.
\end{proof}

\paragraph{\em Proof of Claim \ref{bpfactorization-claim} continued.} Using property \eqref{max-order-xi} we can pick $(\alpha, \beta, \vec \gamma) \in\ed$ such that $\delta_{i'}(\gabgamma) = d_{i'}$ and $\ord_{\xi-c_{i'i'}}\left(\lc_{i'}(\gabgamma)\right) = a_{i'}p_{i',l_{i'}+1} +m_{i',c_{i'i'}}(\vec a)$. Since $\ord_{\xi-c_{i'i'}}(\eta_{i'}(\xi))= m_{i',c_{i'i'}}(\vec a)$, applying Claim \ref{star-claim'} to $h := \gabgamma$ and using property \eqref{max-xi}, we see
\begin{align*}
\deg_\xi\left(\lc_{i'}(\gabgamma)\right)
	&\geq a_{i'}p_{i',l_{i'}+1} + \deg_\xi(\eta_{i'}(\xi)) 
	= \sum_{i \in  \scrN_{i'}} \tilde p_ia_i/p_{i'}
	= m_{i'}(\vec a)
	\geq \deg_\xi\left(\lc_{i'}(\gabgamma)\right).
\end{align*}
It follows that $\deg_\xi\left(\lc_{i'}(\gabgamma)\right) = m_{i'}(\vec a)$, and $\lc_{i'}(\gabgamma) = c(\xi - c_{i'i'})^{a_{i'}p_{i',l_{i'}+1}}\eta_{i'}(\xi)$ for some $c \in \cc^*$. On the other hand, property \eqref{min-xi} and Claim \ref{star-claim'} imply that there is $(\alpha', \beta', \vec \gamma') \in \ed$ such that $\delta_{i'}(g_{\alpha'\beta'\vec\gamma'}) = d_{i'}$ and $\lc_{i'}\left(g_{\alpha'\beta'\vec\gamma'}\right)= c'\eta_{i'}(\xi)\mu(\xi)$ for some $c' \in \cc^*$ and $\mu(\xi) \in \cc[\xi]$ with $\ord_{\xi-c_{i'i'}}(\mu(\xi)) = 0$. It follows that for a generic $\lambda \in \cc^*$, the roots $\xi_{i'j}$ of $\lc_{i'}(\gabgamma + \lambda g_{\alpha'\beta'\vec\gamma'})$ satisfy the base point free condition \eqref{bpf-roots}. Remark-Definition \ref{lc-remark} then implies Claim \ref{bpfactorization-claim} holds for $i'$, and completes the proof of Claim \ref{bpfactorization-claim}.
\end{proof}

Now pick a generic $h \in \Gamma(X, \sheaf(\sum d_j \Gamma_j))$. Claim \eqref{bpfactorization-claim} implies that $h$ has a factorization of the form
\begin{align}
h = g(x,y)\prod_{i=1}^N \prod_{j=1}^{a'_i} \prod_{k=1}^{b_{ij}\tilde p_{i}} \left(y - \psi_{ijk}(x)\right), \label{h-factorization}
\end{align}
where 
\begin{enumerate}
\item $g(x,y) \in \dpsxc[y]$, 
\item $a_i = \sum_{j=1}^{a'_i}b_{ij}$, $1 \leq i \leq N$, and
\item $\psi_{ijk}$'s are conjugates of some $\psi_{ij}$ of the form $\phi_i(x) + \xi_{ij} x^{r_i} + \lot$ for $\xi_{ij} \in \cc^*$ which satisfy the `base point free' condition \eqref{bpf-roots}.
\end{enumerate}
Using the `base point freeness' property of $\xi_{ij}$'s, it is straightforward to compute that for each $i,i',j$, $1 \leq i,i' \leq N$, $1 \leq j \leq a'_i$,
\begin{align*}
& \delta_{i'}\left(\prod_{k=1}^{b_{ij}\tilde p_{i}} (y - \psi_{ijk}(x))\right) = b_{ij}\omega_{i'i},\ \text{so that} \notag \\
& \delta_{i'}(h/g) = \sum_{i=1}^N a_i\omega_{i'i} = d_{i'}.
\end{align*}
It follows that $\delta_{i'}(g) = 0$ for all $i'$, $1 \leq i' \leq N$, and consequently, $g$ is a constant. But then \eqref{h-factorization} and the `base point freeness' property of $\xi_{ij}$'s imply that for generic $h_1, h_2 \in \Gamma(X, \sheaf(\sum d_j \Gamma_j))$, the closures $C_i$, $1 \leq i \leq 2$, on $X$ of the curves defined by $h_i$ on $\cc^2$ intersect distinct and {\em non-singular} points of $X \setminus \cc^2$. This proves \eqref{bpf}, and completes the proof of Theorem \ref{zariski-thm}.
\end{proof}

\section{The case of $X_V$, $V \in \scrV$} \label{onection}
In this section we illustrate Theorem \ref{cox-thm}, Corollary \ref{enricorollary} and Theorem \ref{zariski-thm} for the case of $X_V$ for $V \in \scrV$. We start with the case of $\scrV_1$.\\

Let $V := \cc \Span{G,H^d} \in \scrV_1$. Identify $\pp^2 \setminus \{H=0\}$ with $\cc^2$. Then $g := G|_{\cc^2}$ defines a curve $C$ with one place at infinity. Denote the irreducible components of $X_{V}\setminus \cc^2$ by $\Gamma_0, \ldots, \Gamma_N$, where $\Gamma_0$ is the strict transform of the line at infinity $H=0$. Recall that the dual graph of the curves at infinity are of the form as in Figure \ref{fig1-again}. Recall that there are polynomials $g_{1}, \ldots, g_{s}, g_{s+1} = g$ (the {\em approximate roots} of $g$ introduced by Abhyankar and Moh) such that each $g_k$ is a $\Gamma_{i_k}$-curvette. Choose coordinates $(x,y)$ on $\cc^2$ such that the point at infinity on $C$ is on the closure of the $x$-axis. Then each $g_k$, $1 \leq k \leq s+1$, is a monic polynomial in $y$ (as an element of $\cc[x][y]$). Set $g_0 := x$.

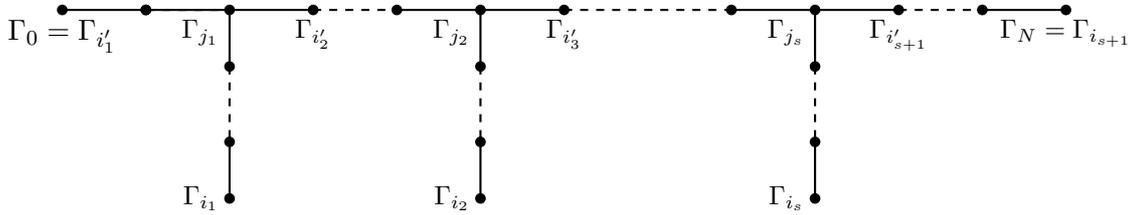
\begin{figure}[htp]

\newcommand{\block}[2]{
	\pgfmathsetmacro\x{\x + \edge}
 	
 	\draw[thick] (\x - \edge,\y) -- (\x,\y);
 	\draw[thick] (\x,\y) -- (\x+\edge,\y);
 	\draw[thick] (\x,\y) -- (\x, \y-\vedge);
 	\draw[thick, dashed] (\x, \y-\vedge) -- (\x, \y-\vedge - \dashedvedge);
 	\draw[thick] (\x, \y-\vedge - \dashedvedge) -- (\x, \y-2*\vedge - \dashedvedge);
 	
 	\fill[black] (\x - \edge, \y) circle (2pt);
 	\fill[black] (\x, \y) circle (2pt);
 	\fill[black] (\x + \edge, \y) circle (2pt);
 	\fill[black] (\x, \y- \vedge) circle (2pt);
 	\fill[black] (\x, \y- \vedge - \dashedvedge) circle (2pt);
 	\fill[black] (\x, \y- 2*\vedge - \dashedvedge) circle (2pt);
 	
 	\draw (\x,\y )  node (center) [below left] {\small{$\Gamma_{j_{#2}}$}};
 	\draw (\x,\y- 2*\vedge - \dashedvedge)  node (bottom3) [left] { \small $\Gamma_{i_{#2}}$};
 	\draw (\x+\edge,\y )  node (center) [below] {\small{$\Gamma_{i'_{#1}}$}};
 	
 	\pgfmathsetmacro\x{\x + \edge}
}

\begin{center}
\begin{tikzpicture}

	\pgfmathsetmacro\x{0}
	\pgfmathsetmacro\y{0}

 	\pgfmathsetmacro\edge{1.1}
 	\pgfmathsetmacro\vedge{.75}
 	\pgfmathsetmacro\dashedvedge{1}
 	
 	\fill[black] (0, 0) circle (2pt);	
 	\draw (0,0)  node (e0) [below] {$\Gamma_0 = \Gamma_{i'_1}$};
 	\draw[thick] (0,0) -- (\edge, 0);
 	\fill[black] (\edge,0) circle (2pt);
 	\draw[thick, dashed] (\edge,0) -- (2*\edge,0);
 	
 	\pgfmathsetmacro\x{\x + \edge}
	\block{2}{1}
	\draw[thick, dashed] (\x,0) -- (\x + \edge,0);
	
	\pgfmathsetmacro\x{\x + \edge}
	\block{3}{2}
	\draw[thick, dashed] (\x,0) -- (\x + 2*\edge,0);
	
	\pgfmathsetmacro\x{\x + 2*\edge}
	\block{s+1}{s}	
	
	\draw[thick, dashed] (\x,0) -- (\x + \edge,0);	
	\fill[black] (\x + \edge,0) circle (2pt);
	\draw[thick] (\x + \edge,0) -- (\x + 2*\edge,0);	
	\fill[black] (\x + 2*\edge,0) circle (2pt);
 	\draw (\x+2*\edge,0)  node [below] {\small $\Gamma_N = \Gamma_{i_{s+1}} $};	
\end{tikzpicture}
\end{center}
\caption{Dual graph of curves at infinity on $X_{V}$ for $V \in \scrV_1$} \label{fig1-again}
\end{figure}
Fix $k$, $0 \leq k \leq N$. Let $l_k$, $0 \leq l_k \leq s$, be such that $\Gamma_k \in [\Gamma_{i'_{l_k+1}}, \Gamma_{i_{l_k+1}}]$ (where $[\Gamma_{i'_{j}}, \Gamma_{i_{j}}]$ is defined as in the proof of Proposition \ref{pencil-surfaces}). Also, let $\delta_k$ be the \semidegree\ on $\cc[x,y]$ associated to $\Gamma_k$.

\begin{rem} \label{puiseux-remark}
The above definition of $l_k$ agrees with the definition in Notation \ref{set-up}, i.e.\ the degree-wise Puiseux series associated to $\delta_k$ has precisely $l_k+1$ {\em formal Puiseux pairs} (Definition \ref{formal-defn}) $(q_{k,1}, p_{k,1}), \ldots,(q_{k,l_k+1}, p_{k,l_k+1})$. Moreover, 
\begin{enumerate}
\item $p_{k,j} = \frac{\gcd\left(\delta_k(g_0), \ldots, \delta_k(g_{j-1})\right)}{\gcd\left(\delta_k(g_0), \ldots, \delta_k(g_{j})\right)}$, $1 \leq j \leq l_k+1$.
\item We can set $f_{kl_k} := g_{l_k}$ (where the correspondence between $f_{kl_k}$ and $\delta_k$ is as in Lemma \ref{f-ij-lemma})
\end{enumerate}
\end{rem}

Recall that $\{g_{1}, \ldots, g_{s+1}\}$ satisfies the properties of $\{f_{ij}\}_{i,j}$ of Theorem \ref{cox-thm} for $X = X_{V}$ (Remark \ref{one-place-f-ij}). For each $k_1, k_2$, $1 \leq k_1 \leq k_2 \leq s+1$, define 
\begin{align}
e_{k_1,k_2} &:= \frac{\deg(g_{k_2})}{\deg(g_{k_1})}.	\label{e}
\end{align}
Let $\scrB := \{i: 0 \leq i \leq N,$ and $\Gamma_i$ is on the (only) `horizontal segment' (which has $\Gamma_0$ and $\Gamma_N$ on the two ends) of the graph in Figure \ref{fig1-again}$\}$. The following lemma, which follows from more or less straightforward computations, computes $\scrN_i, \scrN_{i,c}, c_{ii}$ and $\lc_i$ (from Definition \ref{ndefn} and Remark-Definition \ref{lc-remark}), $1 \leq i \leq N$.

\begin{lemma}
Set $f_{kl_k} := g_{l_k}$, $0 \leq k \leq N$. 
\begin{enumerate}
\item Fix $i,k$, $0 \leq i,k \leq N$. Then $k \in \scrN_i$ iff one of the following conditions holds:
\begin{enumerate}
\item $k = i$.
\item $i \in \scrB$ and the position of $\Gamma_k$ on the graph of Figure \ref{fig1-again} is below and/or to the right of $\Gamma_i$. 
\item $\Gamma_i$ and $\Gamma_k$ are both on the same `vertical segment' of the graph of Figure \ref{fig1-again} and $\Gamma_k$ is below $\Gamma_i$. 
\end{enumerate}
\item 
\begin{enumerate}
\item If $k \not\in \{j_1, \ldots, j_s\}$, then $\scrN_{k,c_{kk}} = \scrN_k\setminus\{k\}$.
\item \label{c^*_q} If $k = j_q$, $1 \leq q \leq s$, then $c_{kk} = 0$ and $\scrN_k = \scrN_{j_q,0} \cup \scrN_{j_q, c^*_q} \cup \{j_q\}$ for some $c^*_q \neq 0$. Moreover, 
\begin{enumerate}
\item $\scrN_{j_q,0}$ consists of all $l$ such that $l \neq j_q$ and the position of $\Gamma_l$ on the graph in Figure \ref{fig1-again} is on the vertical segment between $\Gamma_{j_q}$ and $\Gamma_{i_q}$.
\item $\scrN_{j_q, c^*_q} = \scrN_{j_q} \setminus \left(\scrN_{j_q, 0} \cup \{j_q\}\right)$.
\end{enumerate}
\end{enumerate}
\item
\begin{enumerate}
\item If $k \not\in \scrB$, then $\lc_k(g_1^{\alpha_1} \cdots g_{s+1}^{\alpha_{s+1}}) = c(\xi-c_{kk})^{\alpha_{l_k+1}}$ for some $c \in \cc^*$. If in addition $k \neq i_{l_k+1}$, then $c_{kk} = 0$.
\item If $k \in \scrB\setminus\{j_{1}, j_2, \ldots, j_{s} \}$, then $\lc_k(g_1^{\alpha_1} \cdots g_{s+1}^{\alpha_{s+1}}) = c(\xi - c_{kk})^{\sum_{i =l_k+1}^{s+1} \alpha_{i}e_{l_k+1,i}}$ for some $c \in \cc^*$.
\item If $\Gamma_k = \Gamma_{j_{q}}$ for some $q$, $1 \leq q \leq s$, then $\lc_k(g_1^{\alpha_1} \cdots g_{s+1}^{\alpha_{s+1}}) = c\xi^{\alpha_q}(\xi^{e_{q,q+1}} - c^*_q)^{\sum_{i = q+1}^{s+1} \alpha_{i}e_{q+1,i}}$, where $c \in \cc^*$ and $c^*_q$ is as in assertion \ref{c^*_q}. \qed
\end{enumerate}
\end{enumerate}
\end{lemma}

The following result follows from combining Theorem \ref{cox-thm}, Corollary \ref{enricorollary} and Theorem \ref{zariski-thm} with the above observations (plus Remark \ref{positively-zariski-infinite}). At first we recall some notations: let $\vec e_j$, $0 \leq j \leq N$, be the $j$-th unit vector in $\zz^{N+1}$. For $\vec d := (d_0, \ldots, d_N) \in \zz^N$, write $\Gammad$ for the divisor $\sum_{j=0}^N d_j\Gamma_j \in \cl(X_V)$ and if $f \in \Gamma\left(X_V, \sheaf(\Gammad)\right)$, we denote by $(f)_{\vec d}$ the `copy' of $f$ in the $\vec d$'th graded component of $\scrR(X_V)$. Moreover, for $f \in \cc(x,y)$ we denote by $\vec\delta(f)$ the element $(\delta_0(f), \ldots, \delta_N(f)) \in \zz^{N+1}$. 

\begin{cor} \label{one-place-cor}
\mbox{}
\begin{enumerate}
\item $\vecdelta{g_j}$'s together with $(1)_{\vec e_j}$, $0 \leq j \leq N$, generate $\scrR(X_V)$ as a $\cc$-algebra.
\item Let $\vec d := (d_0, \ldots, d_N) \in \zz^{N+1}$. Let $\scrE_{\vec d}$ be the collection of all $\vec \alpha := (\alpha_0, \ldots, \alpha_{s+1}) \in \zz_{\geq 0}^{s+2}$ such that $\sum_{j=0}^{s+1} \alpha_{j}\delta_k(g_{j}) \leq d_k$ for all $k$, $0 \leq k \leq N$. Then $\Gamma\left(X, \sheaf(\Gammad)\right)$ is a generated as a vector space over $\cc$ by $\{\prod_{j=0}^{s+1} g_j^{\alpha_j}: \vec \alpha\in \ed\}$.
\item Let $\vec d := (d_0, \ldots, d_N) \in \zz^{N+1}$. Define $\scrE'_{\vec d} := \{(\alpha_0,\sum_{j=1}^{s+1}\alpha_j \deg_y(g_j)): (\alpha_0, \ldots, \alpha_{s+1})\in \scrE_{\vec d}\} \subseteq \zz_{\geq 0}^2$. Then $\dim \Gamma\left(X, \sheaf(\Gammad)\right) = |\scrE'_{\vec d}|$. For each $(a,b) \in \scrE'_{\vec d}$, pick any $\vec \alpha \in \ed$ such that $a = \alpha_0$ and $b = \sum_{j=1}^{s+1}\alpha_j \deg_y(g_j)$, and define $g_{a,b} := \prod_{j=0}^{s+1} g_j^{\alpha_j}$. Then $\{g_{a,b}: (a,b) \in \scrE'_{\vec d}\}$ is a basis of $\Gamma\left(X, \sheaf(\Gammad)\right)$.
\item For each $\vec d := (d_0, \ldots, d_N) \in \zz^{N+1}$, the following are equivalent:
\begin{enumerate}
\item $\Gammad \in \pst(X_V)$.
\item $d_j = \max\left\{\sum_{k=0}^{s+1} \alpha_k\delta_j(g_k): (\alpha_0, \ldots, \alpha_{s+1}) \in \ed\right\}$ for all $j$, $0 \leq j \leq N$.	
\item For each $j$, $0 \leq j \leq N$, there exists $(\alpha_0, \ldots, \alpha_{s+1}) \in \zz_{\geq 0}^{s+2}$ such that $\sum_{k=0}^{s+1} \alpha_k\delta_j(g_k) = d_j$ and $\sum_{k=0}^{s+1} \alpha_k\delta_i(g_k) \leq d_i$ for all $i$, $0 \leq i \leq N$.
\end{enumerate}
\item For each $k$, $0 \leq k \leq N$, and each $\vec \alpha := (\alpha_0, \ldots, \alpha_{s+1}) \in \zz_{\geq 0}^{s+2}$, define
\begin{align*}
\mu_k(\vec \alpha) &:= 
	\begin{cases}
		\alpha_{l_k+1} & \text{if}\ k \not\in \scrB \\
		\sum_{i \geq l_k+1} \alpha_{i}e_{l_k+1,i} & \text{if}\ k \in \scrB.
	\end{cases}
&\qquad
\nu_k(\vec \alpha) &:= 
	\begin{cases}
		\alpha_{l_k+1} & \text{if}\ k \not\in \scrB \\
		\sum_{i \geq l_k+1} \alpha_{i}e_{l_k+1,i} & \text{if}\ k \in \scrB\setminus\{j_{1}, \ldots, j_{s} \} \\
		\alpha_q & \text{if}\ k = j_{q},\ 1 \leq q \leq s.
	\end{cases}
\end{align*}
Let $\vec a := (a_0, \ldots, a_N) \in \zz_{\geq 0}^{N+1}$. Define $d_k(\vec a) := \sum_{i=0}^N a_i\omega_{ki}$, $0 \leq k \leq N$, where
\begin{align*}
\omega_{ki} &:= 
	\begin{cases}
		p_{i,l_i+1}\delta_k(f_{il_i}) & \text{if}\ k \not\in \scrN_i\\
		\frac{\prod_{j=1}^{l_k+1}p_{k,j}}{\prod_{j=1}^{l_i}p_{i,j}} \delta_i(f_{il_i})  & \text{if}\ k \in \scrN_i.
	\end{cases}
\end{align*}
and set $\vec d_{\vec a} := (d_0(\vec a), \ldots, d_N(\vec a))$. For each $k$, $0 \leq k \leq N$, define
\begin{alignat*}{2}
m_{k}(\vec a) &:= \sum_{i \in \scrN_{k}} \frac{\prod_{j=1}^{l_i+1}p_{i,j}}{\prod_{j=1}^{l_k}p_{k,j}}a_{i},
&\qquad
n_{k}(\vec a) &:= \sum_{i \in \scrN_{k,c_{kk}}} \frac{\prod_{j=1}^{l_i+1}p_{i,j}}{\prod_{j=1}^{l_k}p_{k,j}}a_{i}.
\end{alignat*} 
Then for each $\vec d \in \zz^{N+1}$, the following are equivalent:
\begin{enumerate}[(a)]
\item $\Gammad \in \tilde P(X_V) $.
\item  $\vec d = \vec d_{\vec a}$ for some $\vec a \in \zz_{\geq 0}^{N+1}$ such that for each $k$, $0 \leq k \leq N$, 
\begin{align*}
d_k &= \max\left\{\sum_{j=0}^{s+1} \alpha_j\delta_k(g_j): \vec \alpha \in \ed\right\}, \\
m_{k}(\vec a) &= \max\left\{ \mu_k(\vec \alpha ): \vec \alpha \in \ed,\ \sum_{j=0}^{s+1} \alpha_j\delta_k(g_j)= d_k \right\}, \\
n_{k}(\vec a) &= \min\left\{ \nu_k(\vec \alpha ): \vec \alpha \in \ed,\ \sum_{j=0}^{s+1} \alpha_j\delta_k(g_j)= d_k \right\}, \\
a_kp_{k,l_k+1} + n_{k}(\vec a) &= \max\left\{ \nu_k(\vec \alpha ): \vec \alpha \in \ed,\ \sum_{j=0}^{s+1} \alpha_j\delta_k(g_j)= d_k \right\}. \qed
\end{align*} 
\end{enumerate}
\end{enumerate}
\end{cor}

Let $V_i := \cc \Span{G_i,H^d} \in \scrV$, $1 \leq i \leq 2$. Identify $\pp^2 \setminus \{H=0\}$ with $\cc^2$. Assume that the germs at infinity of curves defined by $G_1$ and $G_2$ are equisingular. Then there is a natural bijection between the sets of irreducible curves at infinity on $X_{V_i}$'s which induces an isomorphism $\rho: \Pic(X_{V_1}) \cong \Pic(X_{V_2})$. 

\begin{cor} \label{equisingular-corollary}
\mbox{}
\begin{enumerate}
\item \label{equi-dimension} For each line bundle $\scrL  \in \Pic(X_{V_1})$, $\dim_\cc(\Gamma(X_{V_1}), \scrL) = \dim_\cc(\Gamma(X_{V_2}), \sheaf(\rho(\scrL)))$.
\item \label{equi-semigroups} $\rho$ induces isomorphisms $\pst(X_{V_1}) \cong \pst(X_{V_2})$ and $\tilde P_{\infty}(X_{V_1}) \cong \tilde P_{\infty}(X_{V_2})$. In particular, if $D$ is an effective divisor with support at infinity (i.e.\ $\supp(D) \subseteq X_{V_1}\setminus \cc^2$), then $D \in  \tilde P_{\infty}(X_{V_1})$ iff $\rho(D) \in \tilde P_{\infty}(X_{V_2})$.
\item \label{equi-one-place} If $V_1, V_2 \in \scrV_1$, then $\tilde P(X_{V_1}) \cong \tilde P(X_{V_2})$.
\end{enumerate} 
\end{cor}

\begin{proof}
If $V_1, V_2 \in \scrV_1$, then the statements of Corollary \ref{equisingular-corollary} follow from Corollary \ref{one-place-cor} and Remark \ref{positively-zariski-infinite}. The general case follows via almost the same arguments. 
\end{proof}

\section[Appendix]{Appendix - preliminaries on degree-wise Puiseux series} \label{appendix}

\begin{defn}[Degree-wise Puiseux series] \label{dpuiseuxnition}
The field of {\em degree-wise Puiseux series} in $x$ is 
$$\dpsxc := \bigcup_{p=1}^\infty \cc((x^{-1/p})) = \left\{\sum_{j \leq k} a_j x^{j/p} : k,p \in \zz,\ p \geq 1 \right\},$$
where for each integer $p \geq 1$, $\cc((x^{-1/p}))$ denotes the field of Laurent series in $x^{-1/p}$. Let $\phi$ be a degree-wise Puiseux series in $x$. The {\em polydromy order} (terminology taken from \cite{casas-alvero}) of $\phi$ is the smallest positive integer $p$ such that $\phi \in \cc((x^{-1/p}))$. For any $r \in \qq$, let us denote by $[\phi]_{>r}$ (resp.\ $[\phi]_{\geq r}$) sum of all terms of $\phi$ with order greater than (resp.\ greater than or equal to) $r$. Then the {\em Puiseux pairs} of $\phi$ are the unique sequence of pairs of relatively prime integers $(q_1, p_1), \ldots, (q_k,p_k)$ such that the polydromy order of $\phi$ is $p_1\cdots p_k$, and for all $j$, $1 \leq j \leq k$,
\begin{compactenum}
\item $p_j \geq 2$,
\item $[\phi]_{>\frac{q_j}{p_1\cdots p_j}} \in \cc((x^{-\frac{1}{p_0\cdots p_{j-1}}}))$ (where we set $p_0 := 1$), and 
\item $[\phi]_{\geq \frac{q_j}{p_1\cdots p_j}} \not\in \cc((x^{-\frac{1}{p_0\cdots p_{j-1}}}))$.
\end{compactenum}
The exponents $q_j/(p_1 \cdots p_j)$, $1 \leq j \leq k$, are called the {\em characteristic exponents} of $\phi$. Let $\phi = \sum_{q \leq q_0} a_q x^{q/p}$, where $p$ is the polydromy order of $\phi$. Then the {\em conjugates} of $\phi$ are $\phi_j := \sum_{q \leq q_0} a_q \zeta^q x^{q/p}$, $1 \leq j \leq p$, where $\zeta$ is a primitive $p$-th root of unity.
\end{defn}

\begin{defn} \label{formal-defn}
Let $\delta$ be a semidegree and, as in \eqref{puiseux-degree}, let $\tilde \phi(x,\xi):= \phi(x) + \xi x^{r}$ be the {\em generic degree-wise Puiseux series} associated to $\delta$. Let the Puiseux pairs of $\phi$ be $(q_1, p_1), \ldots, (q_l,p_l)$. Express $r$ as $q_{l+1}/(p_1 \cdots p_lp_{l+1})$ where $p_{l+1} \geq 1$ and $\gcd(q_{l+1}, p_{l+1}) = 1$. Then the {\em formal Puiseux pairs} of $\tilde \phi$ are $(q_1, p_1), \ldots, (q_{l+1},p_{l+1})$. Note that 
\begin{enumerate}
\item $\delta(x) = p_1 \cdots p_{l+1}$,
\item it is possible that $p_{l+1} = 1$ (as opposed to other $p_k$'s, which are always $\geq 2$). 
\end{enumerate}
The {\em formal characteristic exponents} of $\tilde \phi$ are $q_j/(p_1 \cdots p_j)$, $1 \leq j \leq l+1$.
\end{defn}

\begin{example} 
Let $(p,q)$ are integers such that $p >0$ and $\delta$ be the weighted degree on $\cc(x,y)$ corresponding to weights $p$ for $x$ and $q$ for $y$. Then $\tilde \phi = \xi x^{q/p}$ (i.e.\ $\phi = 0$). 
\end{example}

The rest of this section is devoted to a proof of Lemma \ref{delta-comparison} below, which is used in the proofs of Theorem \ref{cox-thm} and Theorem \ref{zariski-thm}.

\begin{defn} \label{star-defn}
Let $\phi = \sum_j a_j x^{q_j/p} \in \dpsxc$ be a degree-wise Puiseux series with polydromy order $p$ and $r$ be a multiple of $p$. Then for all $c \in \cc$ we define 
$$c \star_r \phi := \sum_j a_j c^{q_jr/p}x^{q_j/p}.$$   
\end{defn}

\begin{rem} 
Let $\phi, p$ be as in Definition \ref{star-defn}. Then
\begin{enumerate}
\item Every conjugate of $\phi$ is of the form $\zeta^k \star_p \phi$ where $0 \leq k \leq p-1$ and $\zeta$ is a primitive $p$-th root of unity.
\item Let $d$ and $e$ be positive integers, and $c \in \cc$. Then $c \star_{pde} \phi = c^e \star_{pd} \phi = c^{de} \star_p \phi$.
\end{enumerate}
\end{rem}

\begin{lemma} \label{delta-comparison}
Let $\delta$ be a semidegree with generic degree-wise Puiseux series $\tilde \phi(x,\xi) = \phi(x) + \xi x^r$. Let $\psi_1, \psi_2 \in \dpsxc$, and $g_1, g_2 \in \dpsxc[y]$ be defined as 
\begin{align*}
g_i = \prod_{\parbox{1.75cm}{\scriptsize{$\psi_{ij}$ is a con\-ju\-ga\-te of $\psi_i$}}}\mkern-27mu \left(y - \psi_{ij}(x)\right).
\end{align*}
Let $\epsilon_{ij} := \deg_x(\tilde \phi(x,\xi) - \psi_{ij}(x))$, and $\epsilon_i := \min_j \{\epsilon_{ij}\}$. If $\epsilon_1 \geq \epsilon_2$, then 
\begin{align*}
\frac{\delta(g_1)}{\deg_y(g_1)} \geq \frac{\delta(g_2)}{\deg_y(g_2)}.
\end{align*}
\end{lemma}

\begin{proof}
W.l.o.g.\ (replacing $\psi_i$'s by some of their conjugates if necessary) we may assume that $\epsilon_i = \deg_x(\tilde \phi(x,\xi) - \psi_{i}(x))$ for each $i$. Then we can write:
\begin{align*}
\phi(x) &= \phi_1(x) + \phi_2(x) + \phi_3(x, \xi) \\
\psi_1(x) &= \phi_1(x) + \hat \psi_1(x)\\
\psi_2(x) &= \phi_1(x) +\phi_2(x) + \hat \psi_2(x),
\end{align*} 
where all terms of $\phi_{j+1}$ and $\hat \psi_j$ have lower degree in $x$ than each term of $\phi_{j}$ for $j = 1,2$. Let the polydromy orders of $\phi_1$, $\phi_1 + \phi_2$, $\psi_1$, $\psi_2$ be respectively $p_1$, $p_1p_2$, $p_1q_1$, and $p_1p_2q_2$. Let $\zeta_1, \zeta_2, \hat \zeta_1$, $\hat \zeta_2$ be primitive roots of unity of orders $p_1, p_1p_2, p_1q_1, p_1p_2q_2$ respectively. Then 
\begin{align*}
\delta(g_1) 
	&= \delta(x) \sum_{j=0}^{p_1q_1 -1} \deg_x\left(\tilde \phi(x,\xi) - \hat\zeta_1^j \star_{p_1q_1} \psi_1(x)\right)\\
	&= \delta(x) \sum_{j_1=0}^{q_1 -1}\sum_{j=0}^{p_1 -1} \deg_x\left(\phi_1(x) + \phi_2(x) +\xi x^r - 
			{\hat\zeta_1}^{j_1p_1 + j} \star_{p_1q_1}(\phi_1(x) + \hat \psi_1(x))\right)\\
	&= \delta(x) \sum_{j_1=0}^{q_1 -1}\sum_{j=0}^{p_1 -1} \deg_x\left(\phi_1(x) + \phi_2(x) +\xi x^r - 
			\zeta_1^j \star_{p_1}\phi_1(x) + {\hat\zeta_1}^{j_1p_1 + j} \star_{p_1q_1} \hat \psi_1(x)\right)\\
	&= \delta(x) \sum_{j_1=0}^{q_1 -1}\left(\deg_x\left(\phi_2(x) +\xi x^r -  \hat\zeta_1^{j_1p_1} \star_{p_1q_1} \psi_1(x)\right)
		+ \sum_{j=1}^{p_1 -1} \deg_x\left(\phi_1(x) - \zeta_1^j \star_{p_1}\phi_1(x) \right) \right)\\
	&= q_1\delta(x)\left(\epsilon_1 +  \sum_{j=1}^{p_1 -1} \deg_x\left(\phi_1(x) - \zeta_1^j \star_{p_1}\phi_1(x) \right) \right),\\
\delta(g_2) 
	&= \delta(x) \sum_{j=0}^{p_1p_2q_2 -1} \deg_x\left(\tilde \phi(x,\xi) - \hat\zeta_2^j \star_{p_1p_2q_2} \psi_2(x)\right)\\
	&= \delta(x) \sum_{j_2=0}^{p_2q_2 -1}\sum_{j_1=0}^{p_1 -1} \deg_x\left(\phi_1(x) + \phi_2(x) +\xi x^r - 
			{\hat\zeta_1}^{j_2p_1 + j_1} \star_{p_1p_2q_2}(\phi_1(x) + \phi_2(x) + \hat \psi_2(x))\right)\\
	&= \delta(x) \sum_{j_2=0}^{p_2q_2 -1} \left( \deg_x\left(\phi_2(x) +\xi x^r - 
			{\hat\zeta_1}^{j_2p_1} \star_{p_1p_2q_2}(\phi_2(x) + \hat \psi_2(x))\right)
			+ \sum_{j_1=1}^{p_1 -1} \deg_x\left(\phi_1(x) - \zeta_1^{j_1} \star_{p_1}\phi_1(x) \right) \right)\\
	&\leq p_2q_2\delta(x)\left(\epsilon_1 +  \sum_{j=1}^{p_1 -1} \deg_x\left(\phi_1(x) - \zeta_1^j \star_{p_1}\phi_1(x) \right) \right).
\end{align*}
It follows that $\delta(g_2)/(p_1p_2q_2) \leq \delta(g_1)/(p_1q_1)$, as required.
\end{proof}

\bibliographystyle{alpha}
\bibliography{../../utilities/bibi}

\end{document}

%% file: commands-2012.tex
\makeatletter

\newcommand{\Rmnum}[1]{\expandafter\@slowromancap\romannumeral #1@}
\makeatother

\DeclareMathOperator\cl{Cl}

\DeclareMathOperator\ord{ord} 
\DeclareMathOperator\Pic{Pic}

\DeclareMathOperator\supp{Supp}

\newcommand{\scrB}{\ensuremath{\mathcal{B}}}
\newcommand{\scrC}{\ensuremath{\mathcal{C}}}

\newcommand{\scrE}{\ensuremath{\mathcal{E}}}

\newcommand{\scrG}{\ensuremath{\mathcal{G}}}

\newcommand{\scrL}{\ensuremath{\mathcal{L}}}

\newcommand{\scrN}{\ensuremath{\mathcal{N}}}

\newcommand{\scrR}{\ensuremath{\mathcal{R}}}

\newcommand{\scrV}{\ensuremath{\mathcal{V}}}

\newcommand{\cc}{\ensuremath{\mathbb{C}}}

\newcommand{\pp}{\ensuremath{\mathbb{P}}}
\newcommand{\qq}{\ensuremath{\mathbb{Q}}}

\newcommand{\zz}{\ensuremath{\mathbb{Z}}}

\newcommand{\sheaf}{\ensuremath{\mathcal{O}}}

\newcommand{\im}{\ensuremath{\Rightarrow}}

\newcommand{\dsum}{\ensuremath{\bigoplus}}
\newcommand{\into}{\ensuremath{\hookrightarrow}}

\newtheorem{thm}{Theorem}[section]
\newtheorem*{thm*}{Theorem}
\newtheorem{lemma}[thm]{Lemma}
\newtheorem*{lemma*}{Lemma}

\newtheorem{prop}[thm]{Proposition}
\newtheorem*{prop*}{Proposition}
\newtheorem{cor}[thm]{Corollary}

\newtheorem*{claim*}{Claim}
\newtheorem{proclaim}{Claim}[thm]
\newtheorem*{conjecture*}{Conjecture}

\theoremstyle{definition}

\newtheorem*{constrinition*}{Construction-Definition}
\newtheorem{convention}[thm]{Convention}
\newtheorem*{convention*}{Convention}
\newtheorem{defn}[thm]{Definition}
\newtheorem*{defn*}{Definition}

\newtheorem*{definotation*}{Definition-Notation}
\newtheorem{example}[thm]{Example}
\newtheorem*{example*}{Example}

\newtheorem*{fact*}{Fact}

\newtheorem*{facts*}{Facts}

\newtheorem{notation}[thm]{Notation}

\newtheorem*{bold-note*}{Note}
\newtheorem{bold-question}[thm]{Question}
\newtheorem*{bold-question*}{Question}
\newtheorem{rem}[thm]{Remark}
\newtheorem{reminition}[thm]{Remark-Definition}
\newtheorem*{reminition*}{Remark-Definition}

\newtheorem{remexample*}{Remark-Example}

\newtheorem*{remtation*}{Remark-Notation}

\newtheorem*{remuestion*}{Remark-Question}

\theoremstyle{remark}
\newtheorem*{rem*}{Remark}
\newtheorem*{note*}{Note}
\newtheorem*{notation*}{Notation}
\newtheorem*{question*}{Question}

\newtheorem*{questions*}{Questions}

\newcounter{UnorderedProofTempCtr}
\newcommand{\tempcommand}{}